\documentclass[a4paper,12pt,reqno]{amsart} 

\usepackage{amssymb,amsmath,amsfonts,amsthm,mathtools}
\usepackage{latexsym,mathrsfs,pb-diagram}
\usepackage[pdftex]{graphicx}
\usepackage{pgf,tikz,tikz-cd}

\usepackage[hmargin=3.0cm,vmargin=3.0cm]{geometry} 

\usepackage[plainpages=false,colorlinks,pdfpagelabels]{hyperref}
\hypersetup{
  urlcolor=black,
  citecolor=green,
  linkcolor=blue
}

\numberwithin{equation}{section}

\theoremstyle{definition}
\newtheorem{Def}{Definition}[section]

\theoremstyle{remark}

\newtheorem{Rem}[Def]{Remark}

\theoremstyle{plain}
\newtheorem{Prop}[Def]{Proposition}
\newtheorem{Cor}[Def]{Corollary}
\newtheorem{Thm}[Def]{Theorem}
\newtheorem{Lem}[Def]{Lemma}

\newtheorem{MThm}{Theorem}

\newcommand{\dfn}{\doteq}
\newcommand{\st}{ \ ; \ }
\newcommand{\lra}{\longrightarrow}
\newcommand{\sset}{\subset}
\newcommand{\bb}{\bullet}

\newcommand{\eset}{\emptyset}
\newcommand{\Z}{\mathbb{Z}}
\newcommand{\N}{\mathbb{N}}
\newcommand{\R}{\mathbb{R}}

\newcommand{\C}{\mathbb{C}}

\DeclareMathOperator{\Span}{\mathrm{span}}
\DeclareMathOperator{\ran}{\mathrm{ran}}
\DeclareMathOperator{\rank}{\mathrm{rank}}

\newcommand{\del}{\partial}
\newcommand{\dd}{\mathrm{d}}

\newcommand{\TT}{\mathbb{T}}

\newcommand{\D}{\mathscr{D}}
\newcommand{\cinfty}{\mathscr{C}^\infty}
\newcommand{\sob}{\mathscr{H}}
\newcommand{\sol}{\mathscr{S}}
\newcommand{\K}{\mathscr{K}}
\newcommand{\RR}{\mathscr{R}}

\DeclareMathOperator{\T}{T}
\newcommand{\MM}{\mathrm{M}}
\newcommand{\LL}{\mathrm{L}}
\newcommand{\DD}{\mathrm{D}}
\newcommand{\VV}{\mathcal{V}}
\newcommand{\psum}{\sideset{}{'}\sum}
\newcommand{\LLambda}{\underline{\Lambda}}

\newcommand{\gr}{\mathfrak}
\newcommand{\vv}{\mathrm}
\newcommand{\ad}{\mathrm{ad}}

\author{Gabriel Ara\'{u}jo}
\address{Universidade de S{\~a}o Paulo, Brazil}
\email{\texttt{gccsa@icmc.usp.br}}

\author{Igor A.~Ferra}
\address{Universidade Federal de S{\~a}o Carlos, Brazil}
\email{\texttt{igorferra@ufscar.br}}

\author{Max R.~Jahnke}
\address{Universit\"at zu K\"oln}
\email{\texttt{max.jahnke@uni-koeln.de}}

\author{Luis F.~Ragognette}
\address{Universidade Federal de Minas Gerais, Brazil}
\email{\texttt{luisragognette@mat.ufmg.br}}

\thanks{Supported by
  the S{\~a}o Paulo Research Foundation
  (FAPESP, grants~2023/11769-5, 2024/08416-6, 2024/12753-8, 2025/08642-9),
  by Conselho Nacional de Desenvolvimento Cient{\'i}fico e Tecnol{\'o}gico
  (CNPq, grants~163837/2022-8, 170244/2023-7, 404175/2023-6, 303483/2024-5, 308433/2025-4, 308705/2025-4),
  by Funda{\c c}{\~a}o de Amparo {\`a} Pesquisa do Estado de Minas Gerais
  (FAPEMIG, grant~APQ-00020-24),
  and by
  Deutsche Forschungsgemeinschaft
  (DFG, grant~JA~3453/2-1).
}

\keywords{Cohomology, global solvability, invariant operators, Lie groups} 
\subjclass[2020]{35F35 (primary), 58J10, 35R03 (secondary)}

\title[]{Real involutive systems on compact Lie groups}

\begin{document}

\begin{abstract}
  On a compact connected Lie group $G$, we study the global solvability and the cohomology spaces of the differential complex associated with an essentially real involutive structure that is invariant under left translations. We prove that solvability in the first degree of the complex implies solvability in all other degrees, and furnish a converse for this fact under a certain commutativity hypothesis (that always holds when $G$ is a torus). Additionally, it is proved that the solvability holds when the structure comes from the Lie algebra of a closed subgroup of $G$. We also investigate real tube structures when $G$ is the base manifold.
\end{abstract}

\maketitle


\section{Introduction}

Let $\Omega$ be a smooth manifold and $\VV$ an involutive subbundle of $\C T \Omega$, the complexified tangent bundle of $\Omega$. Involutivity means that the Lie bracket of any two sections of $\VV$ (which are, in particular, complex vector fields over $\Omega$) is again a section of $\VV$. Such a situation arises in several instances of interest, but particularly in complex geometry: if $\Omega$ is a complex manifold, and $\VV$ is spanned by the antiholomorphic vector fields; or if $\Omega$ is a real hypersurface in a complex manifold, and $\VV$ is spanned by the antiholomorphic vector fields which are tangent to $\Omega$. The latter situation is the object of study of CR geometry, and is concerned with phenomena such as extendability of holomorphic functions across the hypersurface $\Omega$, thus playing a key role in the function theory of the ambient manifold.

Another situation of this kind, more intimately related to our own, arises in the theory of foliations. In PDE terminology, we will assume that $\VV$ is \emph{essentially real}, meaning simply that $\VV$ is the complexification of an involutive subbundle of $T \Omega$; i.e.,~an integrable distribution on $\Omega$. That distribution, in turn, produces, by means of the Frobenius Theorem, a foliation on $\Omega$ whose leaves are tangent to $\VV$.

In either case, to such an involutive bundle $\VV$, a general abstract construction (see~\cite{bch_iis, treves_has}) gives rise to a complex of vector bundles over $\Omega$ and first-order differential operators between them
\begin{equation}
  \label{eq:treves_complex}
  \begin{tikzcd}
    \cdots \arrow{r}{\dd'} &
    \cinfty(\Omega; \LLambda^{q - 1}) \arrow{r}{\dd'} &
    \cinfty(\Omega; \LLambda^q) \arrow{r}{\dd'} &
    \cinfty(\Omega; \LLambda^{q + 1}) \arrow{r}{\dd'} &
    \cdots
  \end{tikzcd}
\end{equation}
which encodes important information of the underlying geometry. In each of the three cases discussed above, such a construction yields: the Dolbeault complex, in complex geometry; the $\bar{\del}_b$ complex, in CR geometry; and the foliated de Rham complex, in foliation theory.

In the present paper, our objects of interest are the cohomology spaces $H^q_{\dd'}(\cinfty(\Omega))$ of the complex~\eqref{eq:treves_complex} when $\VV$ is essentially real and $\Omega$ is compact. Except when $\VV = \C T \Omega$ -- in which case~\eqref{eq:treves_complex} is just the de Rham complex --, the complex~\eqref{eq:treves_complex} is not elliptic. Therefore, one cannot expect $H^q_{\dd'}(\cinfty(\Omega))$ to be finite dimensional, so we compromise on a more modest goal. Notice that all the spaces in~\eqref{eq:treves_complex} carry natural Fr{\'e}chet topologies, for which the differential operators $\dd'$ are continuous, hence the $H^q_{\dd'}(\cinfty(\Omega))$ inherit a quotient topology, making them topological vector spaces: our task will be to determine necessary or sufficient conditions for $H^q_{\dd'}(\cinfty(\Omega))$ to be Hausdorff.

From the point of view of the global analysis of PDEs, the latter property is central, and is called \emph{global solvability in degree $q$}: by Functional Analysis, it is nothing but closedness of the range of $\dd': \cinfty(\Omega; \LLambda^{q - 1}) \to \cinfty(\Omega; \LLambda^{q})$. At this level of generality, however, such questions are very hard to answer, so we turn our attention to special cases involving symmetries.

\subsection*{Left-invariant structures on compact Lie groups}
The case to which we dedicated most of our work is when the manifold is a Lie group and $\VV$ is invariant under left translations (hence called a \emph{left-invariant subbundle}). 
%
In this context, let $G$ be a compact, connected Lie group and $\VV \sset \C T G$ be an essentially real, left-invariant subbundle. Our first main result is the following:
%
%
\begin{MThm}
  \label{mainthm1}
  If
  \begin{equation}
    \label{eq:firstdprime}
    \dd': \cinfty(G) \lra \cinfty(G; \LLambda^{1})
    \quad \text{has closed range}
  \end{equation}
  then
  \begin{equation*}
    \dd': \cinfty(G; \LLambda^{q - 1}) \lra \cinfty(G; \LLambda^{q})
    \quad \text{has closed range}
  \end{equation*}
  for all $1 \leq q \leq  n \dfn \rank (\VV)$. Moreover, in that case, we have
  \begin{equation}
    \label{eq:iso_scriptK}
    H^{q}_{\dd'} (\cinfty(G)) \cong \K(G) \otimes H^{q}_{\dd'} (\C_G),
    \quad 0 \leq q \leq n,
  \end{equation}
  topologically, where:
  \begin{itemize}
  \item $\K(G)$ is the space of homogeneous solutions of $\VV$: an $f \in \cinfty(G)$ belongs to $\K(G)$ if and only if $\LL f = 0$ whenever $\LL$ is a smooth section of $\VV$; and
  \item $H^{q}_{\dd'} (\C_G)$ is the space of left-invariant classes\footref{f1} in $H^{q}_{\dd'} (\cinfty(G))$.
  \end{itemize}
\end{MThm}
Results regarding global solvability at intermediate levels ($q\ge 2$) are scarce compared to the first level of the complex. Our result was possible thanks to a decomposition at the level of the cohomology groups done in Section~\ref {sec:DecompositionCohomology}. With this decomposition in hands, our result is a consequence of Theorem~\ref{thm:rep_in_K} and Remark~\ref{rem:rep-formula-tensor-K}. 

Properly interpreted, the spirit of Theorem~\ref{mainthm1} can be traced back to~\cite{bp99sys, dm16}, where a single (Diophantine) condition simultaneously governs solvability of \emph{all} degrees of the complex for certain involutive structures on tori called \emph{tube structures} (more on that later on). These two works were the major conceptual inspiration for our results: on a product of tori, by using a normalization process (Section~\ref{subsec:normalization}), they reduce the study of a tube structure to a left-invariant one. Recently, a similar approach was carried out in~\cite{ds26} for the product of torus and a compact Lie group. Here, that ``single (Diophantine) condition'' is encoded in~\eqref{eq:firstdprime}, which always holds for two special, disjoint classes of structures:
\begin{enumerate}
\item When $\VV$ is \emph{globally hypoelliptic}, that is: any distribution $f \in \D'(G)$ such that $\LL f \in \cinfty(G)$ for every $\LL$ section of $\VV$ is already smooth to start with. This is the same as global hypoellipticity of the differential operator $\dd': \cinfty(G) \to \cinfty(G; \LLambda^{1})$, which then has closed range as a consequence of~\cite[Proposition~3.11]{afr22}. In the globally hypoelliptic case, we also show (Corollary~\ref{cor:GH}) that formula~\eqref{eq:iso_scriptK} simplifies further, since $\K(G)$ only contains constant functions, hence
  \begin{equation*}
    H^{q}_{\dd'} (\cinfty(G)) \cong H^{q}_{\dd'} (\C_G),
    \quad 0 \leq q \leq n.
  \end{equation*}
\item When the leaves of the foliation generated by $\VV$ are closed in $G$ (Theorem~\ref{thm:subgroup_closed}). Notice that, since $\VV$ is left-invariant, such leaves are just translates of $H \sset G$, the integral subgroup associated with the Lie algebra $\gr{v}$ of left-invariant sections of $\VV$. In this case, formula~\eqref{eq:iso_scriptK} becomes (Corollary~\ref{cor:subgroup_closed_iso}):
  \begin{equation*}
    H^{q}_{\dd'}(\cinfty(G)) \cong \cinfty(G/H) \otimes H^q_{\mathrm{dR}} (H),
    \quad 0 \leq q \leq n.
  \end{equation*}
\end{enumerate}
  
Our Theorem~\ref{mainthm1} is clearly an extension of (and, incidentally, provide alternative proof for) a classical result in~\cite{ce48} asserting the existence of left-invariant representatives of de Rham cohomology classes on $G$. That work was the third major influence in the present paper, and inspired the results in the Appendix. It is also very much related to a deep question posed by one of the authors~\cite{jahnke23}, in the context of elliptic structures, concerning the existence of such left-invariant representatives in various settings.

When $G$ is a torus, the converse of Theorem~\ref{mainthm1} is known~\cite{bp99sys}. In our setting, Theorem~\ref{thm:abelian-converse} provides a partial converse under the additional assumption that the Lie algebra $\gr{v}$ of left-invariant sections of $\VV$ is Abelian. This condition, however, does not preclude the non-commutativity of $G$ itself.

\subsection*{Tube structures} Let $M$ be a compact, connected manifold, and $\TT^m \dfn \R^m / 2 \pi \Z^m$ be the $m$-dimensional torus. There is a particular way to induce an involutive structure on $\Omega \dfn M \times \TT^m$, called a (corank $m$) \emph{tube structure}, which we now describe. Let $\omega_1, \ldots, \omega_m$ be smooth, closed $1$-forms on $M$, and take $\VV \sset \C T (M \times \TT^m)$ the subbundle whose sections are vector fields annihilated by all the $1$-forms
\begin{equation*}
  \zeta_k \dfn \dd x_k - \omega_k,
  \quad k \in \{1, \ldots, m\};
\end{equation*}
here, $x = (x_1, \ldots, x_m)$ are the usual coordinates on $\TT^m$. The global solvability of the differential complex~\eqref{eq:treves_complex} associated with a tube structure is a question of great interest (see for instance~\cite{afjr24, bcp96, bpzz18, hz17, hz19} or~\cite{ckmt26} for non-compact manifolds).

If $M$ is a Lie group (with Lie algebra $\gr{m}$) then so is $G \dfn M \times \TT^m$ (with Lie algebra $\gr{g} = \gr{m} \oplus \R^m$), and $\VV$ is essentially real and left-invariant if and only if $\omega_1, \ldots, \omega_m$ are real-valued and left-invariant. While this is still not quite in the context of Theorem~\ref{mainthm1}, we prove in Corollary~\ref{cor:normalization_lie_groups} that every real tube structure on $G$ is equivalent to a left-invariant one by means of a diffeomorphism. Since the differential complex~\eqref{eq:treves_complex} is an invariant functorially attached to $\VV$, by means of this ``normalization'' we can replace the original tube structure by a much simpler one (of ``constant coefficients'', in a certain sense), and take full advantage of the symmetries of $G$ to analyze it: for now the whole machinery of Theorem~\ref{mainthm1} is available.

We exploit this to investigate a question raised in a previous paper~\cite{afjr24}. There, while looking for a certain representation formula for $H^q_{\dd'}(\cinfty(\Omega))$ for tube structures (akin to the one presented in~\cite{dm16} when $M$ is a torus), we found an obstruction, which vanishes in certain cases (e.g.,~the torus), but not always: we then conjectured that the obstruction should vanish when $M$ is a Lie group (hence, the said representation formula does hold). 

In the present paper, our last main result is the proof of this conjecture (Corollary~\ref{cor:mysterious_cohomology_left_invariant_lis}). These questions are explained in details in Section~\ref{sec:formula-tube}, and settled in Section~\ref{sec:perturb} via perturbations of the de Rham complex by closed $1$-forms.

\section{Preliminaries}
\label{secPreliminaries}

We assume that $G$ is compact and connected Lie group and recall that  an involutive subbundle $\VV \sset \C T G$ is \emph{left-invariant} if 
\begin{equation*}
  (\ell_x)_* \VV_y \sset \VV_{xy},
  \quad \forall x,y \in G;
\end{equation*}
here, $\ell_x: G \to G$ stands for the left translation by $x \in G$. Such bundles are in one-to-one correspondence with complex Lie subalgebras $\gr{v}$ of $\C \gr{g}$ -- the Lie algebra of all complex, left-invariant vector fields on $G$ -- in a natural way. Here, we are interested in the case when $\VV$ is \emph{essentially real}\footnote{This is not necessary in this section, though.}, meaning that $\VV_x = \bar{\VV}_x$ for every $x \in G$ which is equivalent to say that $\gr{v} = \bar{\gr{v}}$ as subspaces of $\C \gr{g}$.

The latter means that we can choose $\LL_1, \ldots, \LL_n$ \emph{real} left-invariant vector fields forming a basis for $\gr{v}$. We choose additional vectors $\vv{M}_1, \ldots, \vv{M}_m \in \C \gr{g}$ such that, together, they form a basis for $\C \gr{g}$. We denote the dual of $\C \gr{g}$ by $\C \gr{g}^*$ and the dual basis of $\vv{L}_1, \ldots, \vv{L}_n, \vv{M}_1, \ldots, \vv{M}_m \in \C \gr{g}$ by $\tau_1, \ldots, \tau_n, \zeta_1, \ldots, \zeta_m \in \C \gr{g}^*$. These are left-invariant $1$-forms on $G$, and, since $m + n = \dim G$, for $f \in \cinfty(G)$ we can write
\begin{equation*}
  \dd f = \sum_{j = 1}^n (\vv{L}_j f) \ \tau_j + \sum_{k = 1}^m (\vv{M}_k f) \ \zeta_k.
\end{equation*}

For $q \in \{0, \ldots, n\}$, let $\cinfty(G; \LLambda^{q})$ be the space of smooth $q$-forms on $G$ of type
\begin{equation}
  \label{eq:urepglobal}
  u = \psum_{|J| = q} u_{J} \ \tau_J,
  \quad u_{J} \in \cinfty(G).
\end{equation}
The prime indicates that the summation is restricted to ordered multi-indices $J = (j_1, \ldots, j_q)$, where we take $\tau_J \dfn \tau_{j_1} \wedge \cdots \wedge \tau_{j_q}$. The exterior derivative induces a differential operator $\dd': \cinfty(G; \LLambda^{q}) \lra \cinfty(G; \LLambda^{q + 1})$ as follows: given $u \in \cinfty(G; \LLambda^{q})$ as in~\eqref{eq:urepglobal} we have
\begin{equation*}
  \dd u
  = \sum_{j = 1}^n \psum_{|J| = q} (\vv{L}_j u_{J}) \ \tau_j \wedge \tau_J
  + \sum_{k = 1}^m \psum_{|J| = q} (\vv{M}_k u_{J}) \ \zeta_k \wedge \tau_J
  + \psum_{|J| = q} u_{J} \ \dd \tau_J,
\end{equation*}
which we project back onto $\cinfty(G; \LLambda^{q + 1})$ by deleting any term containing a $\zeta_k$; for instance, the whole second double sum. As for the third sum, since $\dd \tau_J$ is a left-invariant $(q + 1)$-form on $G$, it can be written as a linear combination over $\C$ of our basic forms, and we only care about the portions of such a combination which happen to live in $\LLambda^{ q + 1}$. That is, there are constants $\alpha_{JK} \in \C$ such that
\begin{equation}
  \label{eq:structure-ctes}
  \dd \tau_J = \psum_{|K| = q + 1} \alpha_{JK} \ \tau_K
  + \text{terms containing at least one $\zeta_k$},
\end{equation}
leading us to
\begin{equation}
  \label{eq:isom_dprime3}
  \dd' u
  \dfn
  \sum_{j = 1}^n \psum_{|J| = q} (\vv{L}_j u_{J}) \ \tau_j \wedge \tau_J
  +
  \psum_{|J| = q} \psum_{|K| = q + 1} \alpha_{JK} u_{J} \ \tau_K.
\end{equation}
We have that $\dd' \circ \dd' = 0$, i.e.,~a differential complex is defined: our purpose is to investigate the cohomology spaces
\begin{equation*}
  H^{q}_{\dd'} (\cinfty(G))
  \dfn
  \frac{ \ker \{ \dd': \cinfty(G; \LLambda^{q}) \lra \cinfty(G; \LLambda^{q + 1}) \} }
  { \ran \{ \dd': \cinfty(G; \LLambda^{q - 1}) \lra \cinfty(G; \LLambda^{q}) \} },
  \quad q \in \{0, \ldots, n\}.
\end{equation*}

A useful (and more invariant) point of view is to regard a $q$-form $w$ on $G$ as an alternating $\cinfty(G)$-multilinear map $w: \cinfty(G; \C TG)^q \to \cinfty(G)$. In that sense, we define $\cinfty(G; \T'^{1, q - 1})$ as the space of all such $w$ satisfying
\begin{equation*}
  w(\vv{X}_1, \ldots, \vv{X}_q) = 0
  \quad \text{whenever $\vv{X}_1, \ldots, \vv{X}_q$ are sections of $\VV$}.
\end{equation*}
One checks at once that every equivalence class in $\cinfty(G; \wedge^q \C T^* G) / \cinfty(G; \T'^{1, q - 1})$ has a unique representative in $\cinfty(G; \LLambda^{q})$; that these two spaces are therefore isomorphic; and that $\dd'$ is, in this regard, precisely the map induced on the quotients by the exterior derivative. Actually, in our concrete setup $\cinfty(G; \T'^{1, q - 1})$ is exactly our set of remainders (``terms containing at least one $\zeta_k$''), hence
\begin{equation*}
  \cinfty(G; \wedge^q \C T^* G) = \cinfty(G; \LLambda^{q}) \oplus \cinfty(G; \T'^{1, q - 1}).
\end{equation*}
Thanks to the invariant formula for the exterior derivative~\cite[Proposition~12.19]{lee_smooth}, for any $u \in \cinfty(G; \LLambda^{q})$ the following holds: if $\vv{X}_0, \vv{X}_1, \ldots, \vv{X}_q$ are sections of $\VV$ then
\begin{multline}
  \label{eq:invariant_formula_dprime}
  (\dd' u)(\vv{X}_0, \ldots, \vv{X}_q)
  = \sum_{k = 0}^q (-1)^{k} \vv{X}_k \left( u(\vv{X}_0, \ldots, \widehat{\vv{X}}_k, \ldots, \vv{X}_q) \right) \\
  + \sum_{j < k} (-1)^{j + k}
  u ([\vv{X}_j, \vv{X}_k], \vv{X}_0, \ldots, \widehat{\vv{X}}_j, \ldots, \widehat{\vv{X}}_k, \ldots, \vv{X}_q).
\end{multline}

Throughout the paper, we extend any linear map $A: \cinfty(G) \to \cinfty(G)$ to forms~\eqref{eq:urepglobal} by letting it act coefficientwise:
\begin{equation*}
  Au \dfn \psum_{|J| = q} (Au_J) \ \tau_J.
\end{equation*}
Regarding $q$-forms on $G$ as maps $\cinfty(G; \C TG)^q \to \cinfty(G)$, this is the same as the composition $A \circ u$ since the coefficients in~\eqref{eq:urepglobal} are determined by $u_J = u(\vv{L}_{j_1}, \ldots, \vv{L}_{j_q})$ for each $J = (j_1, \ldots, j_q)$. In particular, this extension is independent of the choice of basis in~\eqref{eq:urepglobal}. Moreover, formula~\eqref{eq:invariant_formula_dprime} implies that
\begin{equation}
  \label{eq:comm_A_dprime}
  \left( \dd' (A u) - A (\dd' u) \right) (\vv{X}_0, \ldots, \vv{X}_q)
  = \sum_{k = 0}^q (-1)^{k} [\vv{X}_k, A] \left( u(\vv{X}_0, \ldots, \widehat{\vv{X}}_k, \ldots, \vv{X}_q) \right), 
\end{equation}
whatever $u \in \cinfty(G; \LLambda^{q})$ and $\vv{X}_0, \ldots, \vv{X}_q$ sections of $\VV$. Hence, if $A$ commutes with every $\vv{L} \in \gr{v}$ on functions, then its extension also commutes with $\dd'$ on forms.

\subsection{ Convenient metrics, frames and their sub-Laplacians}

We endow $G$ with a Riemannian metric that is $\ad$-invariant, meaning that it is left-invariant and the inner product $\langle \cdot, \cdot \rangle$ induced on $\gr{g}$ turns the linear endomorphism $\vv{Y} \in \gr{g} \mapsto [\vv{X}, \vv{Y}] \in \gr{g}$ into a skew-symmetric map for every $\vv{X} \in \gr{g}$. Compactness of $G$ ensures that metrics with this property always exist, a consequence of~\cite[Proposition~4.24]{knapp_lgbi}. We can now further assume $\LL_1, \ldots, \LL_n$ orthonormal with respect to $\langle \cdot, \cdot \rangle$.
\begin{Lem}
  \label{lem:lap_commutes}
  The second-order operator on $G$
  \begin{equation}
    \label{eq:DeltaV}
    \Delta_{\gr{v}} \dfn - \sum_{j = 1}^n \vv{L}_j^2
  \end{equation}
  commutes with every $\vv{L} \in \gr{v}$.
\end{Lem}
The proof is easy and purely algebraic; we leave it to Appendix~\ref{ap:algebraic_proof}.
\begin{Def}
  A linear subspace $\mathscr{F}(G) \sset \cinfty(G)$ is \emph{$\gr{v}$-invariant} if $\vv{L} f \in \mathscr{F}(G)$ whenever $f \in \mathscr{F}(G)$ and $\vv{L} \in \gr{v}$. In that case, we let
  \begin{equation*}
    \mathscr{F}(G; \LLambda^{q})
    \dfn
    \{ u \in \cinfty(G; \LLambda^{q}) \st u_J \in \mathscr{F}(G), \ \forall J \}.
  \end{equation*}
  It is clear from~\eqref{eq:isom_dprime3} that $\dd' \mathscr{F}(G; \LLambda^{q}) \sset \mathscr{F}(G; \LLambda^{q + 1})$, and we denote for $0 \leq q \leq n$:
  \begin{equation}
    \label{eq:cohomology_F}
    H^{q}_{\dd'} (\mathscr{F}(G))
    \dfn
    \frac{ \ker \{ \dd': \mathscr{F}(G; \LLambda^{q}) \lra \mathscr{F}(G; \LLambda^{q + 1}) \} }
         { \ran \{ \dd': \mathscr{F}(G; \LLambda^{q - 1}) \lra \mathscr{F}(G; \LLambda^{q}) \} }.
  \end{equation}
\end{Def}
E.g.,~the space $\K(G)$ of smooth \emph{homogeneous solutions} of $\VV$ is $\gr{v}$-invariant:
\begin{equation*}
  \K(G)
  \dfn
  \left\{ f \in \cinfty(G) \st \text{$\vv{L}f = 0$ for every $\vv{L}$ section of $\VV$} \right\}.
\end{equation*}

Each $\vv{L}_j$ is left-invariant and real, and as such is skew-symmetric with respect to the $L^2$ inner product induced on functions by the metric. Therefore:
\begin{equation}
  \label{eq:delta-vs-fields-norm}
  \langle \Delta_{\gr{v}} f, f \rangle_{L^2(G)} = \sum_{j = 1}^n \| \vv{L}_j f \|_{L^2(G)}^2,
  \quad \forall f \in \cinfty(G).
\end{equation}
This identity easily implies that
\begin{equation*}
  \K(G) = \ker \{\Delta_{\gr{v}}: \cinfty(G) \lra \cinfty(G) \}.
\end{equation*}

On time, we state a technical result which is key to our further developments. It is also of algebraic nature, hence relies on no further hypotheses; we prove it in more precise form in Proposition~\ref{prop:ce}.
\begin{Lem}
  \label{lem:ce}
  If $g \in \cinfty(G; \LLambda^{q})$ is such that $\dd' g = 0$ then there exists $v \in \cinfty(G; \LLambda^{q - 1})$ such that $\dd' v = \Delta_{\gr{v}} g$. Moreover, for any $\gr{v}$-invariant subspace $\mathscr{F}(G) \sset \cinfty(G)$, if $g \in \mathscr{F}(G; \LLambda^{q})$ then one can take $v \in \mathscr{F}(G; \LLambda^{q - 1})$.
\end{Lem}

\section{Decompositions of forms}
\label{sec:decompositions}

We recall that the following are equivalent~\cite[Proposition~3.11]{afr22}:
\begin{enumerate}
\item $\dd': \cinfty(G) \to \cinfty(G; \LLambda^{1})$ has closed range;
\item $\Delta_{\gr{v}}: \cinfty(G) \to \cinfty(G)$ has closed range.
\end{enumerate}
This motivates the introduction of the space
\begin{equation}
  \label{eq:Delta_range}
  \RR(G)
  \dfn
  \ran \left\{ \Delta_{\gr{v}}: \cinfty(G) \lra \cinfty(G) \right\},
\end{equation}
which is $\gr{v}$-invariant thanks to Lemma~\ref{lem:lap_commutes}. Its closure in $\cinfty(G)$, which we will denote by $\bar{\RR}(G)$, is also $\gr{v}$-invariant: this is easily seen using again Lemma~\ref{lem:lap_commutes} and the continuity of the action of vector fields.
\begin{Prop}
  \label{prop:direct_sum_closed_range}
  We have $\cinfty(G) = \K(G) \oplus \bar{\RR}(G)$. The projections
  \begin{equation*}
    \pi_{\K}: \cinfty(G) \lra \K(G),
    \quad
    \pi_{\bar{\RR}}: \cinfty(G) \lra \bar{\RR}(G),
  \end{equation*}
  are continuous.
\end{Prop}
\begin{proof}
  It is clear that
  \begin{equation}
    \label{eq:sublap_symm}
    \langle \Delta_{\gr{v}} f, g \rangle_{L^2(G)} = \langle f, \Delta_{\gr{v}} g \rangle_{L^2(G)},
    \quad \forall f, g \in \cinfty(G),
  \end{equation}
  hence $\K(G) \perp \RR(G)$ in $L^2(G)$. Therefore, $\K(G) \perp \bar{\RR}(G)$ as well: we must prove that their sum is $\cinfty(G)$.
  
  Let $\Delta_G \dfn \dd^* \dd$ be the Laplace-Beltrami operator, acting on functions, associated to our metric on $G$. The following claims regarding its spectral theory are well-known: the spectrum $\sigma(\Delta_G)$ is countable; the eigenspaces $E_\lambda \sset \cinfty(G)$ associated to eigenvalues $\lambda \in \sigma(\Delta_G)$ are finite dimensional; and 
  \begin{equation}
    \label{eq:fourier_DeltaG}
    f = \sum_{\lambda \in \sigma(\Delta_G)} \mathcal{F}_\lambda(f)
  \end{equation}
  for every $f \in L^2(G)$, where $\mathcal{F}_\lambda: L^2(G) \to E_\lambda$ denotes the orthogonal projection. In a Paley-Wiener-like fashion -- a consequence of Weyl's law -- smoothness of $f$ is characterized by the following condition: for each  $s > 0$ there exists $C_s > 0$ such that
  \begin{equation}
    \label{eq:paley-wiener}
    \| \mathcal{F}_\lambda(f) \|_{L^2(G)} \leq C_s (1 + \lambda)^{-s},
    \quad \forall \lambda \in \sigma(\Delta_G),
  \end{equation}
  in which case the series~\eqref{eq:fourier_DeltaG} actually converges in $\cinfty(G)$.

  Since every $\vv{X} \in \gr{g}$ is skew-symmetric with respect to the $L^2$ inner product on functions, one easily shows that $\Delta_G = \Delta_{\gr{g}}$ (for any choice of orthonormal basis of $\gr{g}$), hence $\Delta_G$ commutes with any $\vv{X} \in \gr{g}$ (Lemma~\ref{lem:lap_commutes}), and then also with $\Delta_{\gr{v}}$. In particular, $\Delta_{\gr{v}} (E_\lambda) \sset E_\lambda$ for every $\lambda \in \sigma(\Delta_G)$: we define
  \begin{equation*}
    \K_\lambda \dfn \ker \left\{ \Delta_{\gr{v}}: E_\lambda \lra E_\lambda \right\},
    \quad
    \RR_\lambda \dfn \ran \left\{ \Delta_{\gr{v}}: E_\lambda \lra E_\lambda \right\}.
  \end{equation*}
  From~\eqref{eq:sublap_symm} we further infer that $E_\lambda = \K_\lambda \oplus \RR_\lambda$ as a sum of orthogonal, finite dimensional subspaces.

  Given $f \in \cinfty(G)$, using the $L^2$-orthogonality we are entitled to decompose uniquely $\mathcal{F}_\lambda(f) = \mathcal{F}_\lambda(f)_{\K} + \mathcal{F}_\lambda(f)_{\RR}$ in $\K_\lambda \oplus \RR_\lambda$, for each $\lambda \in \sigma(\Delta_G)$. It follows from the Paley-Wiener Theorem above that both series
  \begin{equation*}
    f_{\K} \dfn \sum_{\lambda \in \sigma(\Delta_G)} \mathcal{F}_\lambda(f)_{\K},
    \quad 
    f_{\bar{\RR}} \dfn \sum_{\lambda \in \sigma(\Delta_G)} \mathcal{F}_\lambda(f)_{\RR}
  \end{equation*}
  converge in $\cinfty(G)$. Clearly, $f = f_{\K} + f_{\bar{\RR}}$ and $f_{\K} \perp f_{\bar{\RR}}$. On the one hand, it is clear that
  \begin{equation*}
    \Delta_{\gr{v}} (f_{\K})
    =
    \Delta_{\gr{v}} \Big( \sum_{\lambda \in \sigma(\Delta_G)} \mathcal{F}_\lambda(f)_{\K} \Big)
    =
    \sum_{\lambda \in \sigma(\Delta_G)} \Delta_{\gr{v}} \mathcal{F}_\lambda(f)_{\K}
    = 0
  \end{equation*}
  i.e.,~$f_{\K} \in \K(G)$. On the other hand, by definition of $\RR_\lambda$, we can solve
  \begin{equation*}
    \Delta_{\gr{v}} u_\lambda = \mathcal{F}_\lambda(f)_{\RR},
    \quad u_\lambda \in E_\lambda,
  \end{equation*}
  for each $\lambda \in \sigma(\Delta_G)$, hence
  \begin{equation*}
    f_{\bar{\RR}}
    =
    \sum_{\lambda \in \sigma(\Delta_G)} \mathcal{F}_\lambda(f)_{\RR}
    =
    \lim_{\nu \to \infty}
    \sum_{\substack{\lambda \in \sigma(\Delta_G) \\ |\lambda| \leq \nu}} \Delta_{\gr{v}} u_\lambda
  \end{equation*}
  with convergence in $\cinfty(G)$: we conclude that $f_{\bar{\RR}} \in \bar{\RR}(G)$, since as each truncated sum above belongs to $\RR(G)$.
  
  Finally, to check continuity of the projections $\pi_{\K}, \pi_{\bar{\RR}}$ (we endow $\K(G), \bar{\RR}(G)$ with their subspace topologies inherited from $\cinfty(G)$, which turn them into Fr{\'e}chet spaces), we make use of adapted Sobolev norms: for $f \in \cinfty(G)$ and $k \in \Z_+$, let
  \begin{equation*}
    \| f\|_{\sob^k(G)} \dfn \| (I + \Delta_G)^k f \|_{L^2(G)}
    =
    \Big(  \sum_{\lambda \in \sigma(\Delta_G)} (1 + \lambda)^{2k} \| \mathcal{F}_\lambda(f) \|_{L^2(G)}^2 \Big)^{\frac{1}{2}};
  \end{equation*}
  they generate the topology of $\cinfty(G)$ and, given $\mathscr{F} \in \{ \K, \bar{\RR} \}$, it is clear that
  \begin{equation*}
    \| \pi_{\mathscr{F}} (f) \|_{\sob^k(G)}^2
    \leq \| \pi_{\K} (f) \|_{\sob^k(G)}^2 +  \| \pi_{\bar{\RR}} (f) \|_{\sob^k(G)}^2
    = \| f \|_{\sob^k(G)}^2,
  \end{equation*}
  from which continuity of $\pi_{\mathscr{F}}$ is immediate.
\end{proof}

We extend the previous result to forms:
\begin{Cor}
  \label{cor:direct_sum_closed_range_q}
  We have $\cinfty(G; \LLambda^{q}) = \K(G; \LLambda^{q}) \oplus \bar{\RR}(G; \LLambda^{q})$ as a sum of closed subspaces. The associated projections $\pi_{\K}, \pi_{\bar{\RR}}$ are continuous, and commute with $\dd'$.
\end{Cor}
\begin{proof}
  It is clear that $\K(G; \LLambda^{q}) \cap \bar{\RR}(G; \LLambda^{q}) = \{ 0 \}$ due to their coefficientwise definition. We write $u \in \cinfty(G; \LLambda^{q})$ as
  \begin{equation*}
    u
    = \psum_{|J| = q} u_J \ \tau_J
    = \psum_{|J| = q} \left( \pi_{\K} (u_J) + \pi_{\bar{\RR}} (u_J) \right)  \ \tau_J,
  \end{equation*}
  hence $u = \pi_{\K} (u) + \pi_{\bar{\RR}} (u)$, where
  \begin{equation*}
    \pi_{\mathscr{F}} (u) \dfn
    \psum_{|J| = q} \pi_{\mathscr{F}} (u_J) \tau_J \in \mathscr{F}(G; \LLambda^{q}),
    \quad \mathscr{F} \in \{ \K, \bar{\RR} \}.
  \end{equation*}
  All the claims are now obvious. For the last one, notice that $\dd' u = \dd' \pi_{\K} (u) + \dd' \pi_{\bar{\RR}} (u)$, where
  \begin{equation*}
    \dd' \pi_{\mathscr{F}} (u) \in \mathscr{F}(G; \LLambda^{q + 1}),
    \quad \mathscr{F} \in \{ \K, \bar{\RR} \},
  \end{equation*}
  as both $\K(G), \bar{\RR}(G)$ are $\gr{v}$-invariant spaces. Uniqueness in the direct sum decomposition $\cinfty(G; \LLambda^{q + 1}) = \K(G; \LLambda^{q + 1}) \oplus \bar{\RR}(G; \LLambda^{q + 1})$ shows that $\dd' \pi_{\mathscr{F}} (u) = \pi_{\mathscr{F}} (\dd' u)$.
\end{proof}

\section{Decompositions in cohomology}
\label{sec:DecompositionCohomology}

Given two $\gr{v}$-invariant subspaces $\mathscr{F}(G) \sset \mathscr{F}^\bb(G) \sset \cinfty(G)$, the continuous inclusion maps $\mathscr{F}(G; \LLambda^q) \hookrightarrow \mathscr{F}^\bb(G; \LLambda^q)$ define chain maps, hence induce linear maps in cohomology (recall their definition in~\eqref{eq:cohomology_F})
\begin{equation*}
  H^q_{\dd'}(\mathscr{F}(G)) \lra H^q_{\dd'}(\mathscr{F}^\bb(G)),
  \quad q \in \{0, \ldots, n\},
\end{equation*}
continuous for the quotient topologies; they may be neither injective nor surjective.

For $\mathscr{F} \in \{ \K, \bar{\RR} \}$, we constructed in the previous section projections $\pi_{\mathscr{F}}: \cinfty(G; \LLambda^q) \to \mathscr{F}(G; \LLambda^q)$, which are also continuous chain maps, and hence induce continuous linear maps in cohomology:
\begin{equation*}
  H^q_{\dd'}(\cinfty(G)) \lra H^q_{\dd'}(\mathscr{F}(G)),
  \quad q \in \{0, \ldots, n\}.
\end{equation*}
Since the composition
\begin{equation*}
  \begin{tikzcd}
    \mathscr{F}(G; \LLambda^q) \arrow[hookrightarrow]{r} &
    \cinfty(G; \LLambda^q) \arrow{r}{\pi_{\mathscr{F}}} &
    \mathscr{F}(G; \LLambda^q)
  \end{tikzcd}
\end{equation*}
is the identity map, so is the induced composition
\begin{equation}
  \label{eq:diagram-proj-cohom}
  \begin{tikzcd}
    H^q_{\dd'}(\mathscr{F}(G)) \arrow{r} &
    H^q_{\dd'}(\cinfty(G)) \arrow{r} &
    H^q_{\dd'}(\mathscr{F}(G)),
  \end{tikzcd}
\end{equation}
thus making the first (resp.~second) arrow in~\eqref{eq:diagram-proj-cohom} injective (resp.~surjective). In this sense, we regard both $H^q_{\dd'}(\K(G))$, $H^q_{\dd'}(\bar{\RR}(G))$ as linear subspaces of $H^q_{\dd'}(\cinfty(G))$: algebraically, we have the decomposition
\begin{equation}
  \label{eq:direct_sum_cohomology}
  H^q_{\dd'}(\cinfty(G)) \cong H^q_{\dd'}(\K(G)) \oplus H^q_{\dd'}(\bar{\RR}(G))
\end{equation}
whose factors we now investigate.

Consider\footnote{\label{f1} Here, $\C_G$ stands for the space of constant functions on $G$, and $\C_G(G; \LLambda^{q})$ for the (finite dimensional) space of left-invariant forms $u$ in $\cinfty(G; \LLambda^{q})$; i.e., those whose coefficients $u_J$ in representation~\eqref{eq:urepglobal} are all \emph{constant}.} the linear map $\mathsf{T}: \C_G(G; \LLambda^{q}) \to \C_G(G; \LLambda^{q+1})$ defined by
\begin{equation*}
  \mathsf{T} \bigg( \psum_{|J| = q} c_J \ \tau_J \bigg)
  \dfn \psum_{|K| = q + 1} \bigg( \psum_{|J| = q}  \alpha_{JK} c_{J} \bigg) \ \tau_K
\end{equation*}
where $\alpha_{JK} \in \C$ are given by~\eqref{eq:structure-ctes}, i.e.:
\begin{equation}
  \label{eq:structure-ctes2}
  \dd' \tau_J = \psum_{|K| = q + 1} \alpha_{JK} \ \tau_K.
\end{equation}
By standard arguments of Linear Algebra, one can find a linear map
\begin{equation*}
  \mathsf{S}: \C_G(G; \LLambda^{q+1}) \lra \C_G(G; \LLambda^{q})
\end{equation*}
such that $\mathsf{T} \circ \mathsf{S}$  is the identity on $\ran \mathsf{T}$.
\begin{Prop}
  \label{prop:algebraic_part_closed_range}
  The range of $\dd': \K(G; \LLambda^{q}) \to \K(G; \LLambda^{q + 1})$   is always closed. That is: $H^q_{\dd'}(\K(G))$ is Hausdorff (hence, a Fr{\'e}chet space) for each $q \in \{0, \ldots, n\}$.
\end{Prop}
\begin{proof}
  By means of~\eqref{eq:urepglobal}, we interpret elements in $\cinfty(G; \LLambda^q)$ as smooth functions on $G$ valued on the finite dimensional vector space $\C_G(G; \LLambda^{q})$. In that sense, we note that if $u \in  \K(G; \LLambda^{q})$ then for every $x\in G$ have
  \begin{equation}
    \label{imagemdedlinha}
    (\dd' u)(x)
    =  \psum_{|J| = q} \psum_{|K| = q + 1} \alpha_{JK} u_{J}(x) \ \tau_K
    = \mathsf{T}(u(x)).
  \end{equation}

  Now let  $\{u_\nu\}_{\nu \in \N} \sset \K(G; \LLambda^{q})$  be such that $\dd' u_\nu \to f \in \K(G; \LLambda^{q + 1})$. By~\eqref{eq:isom_dprime3}:
  \begin{equation}
    \label{eq:closedness_algebraic_1}
    \psum_{|K| = q + 1} \bigg( \psum_{|J| = q}  \alpha_{JK} (u_\nu)_{J} \bigg) \ \tau_K
    \lra \psum_{|K| = q + 1} f_K \ \tau_K
    \quad
    \text{in $\cinfty(G; \LLambda^{q + 1})$}.
  \end{equation}
  If we think pointwise, a consequence  of~\eqref{eq:closedness_algebraic_1} is that, for every  $x \in G$ fixed,
  \begin{equation*}
    \mathsf{T} (u_\nu(x)) \lra f(x)
    \quad
    \text{in $\C_G(G; \LLambda^{q+1})$}.
  \end{equation*}
  By finite dimensionality of $\C_G(G; \LLambda^{q+1})$,  we conclude that $f(x) \in \ran \mathsf{T}$ for every $x \in G$. For each   $x \in G$  we define $u(x) \dfn \mathsf{S}(f(x)) \in \C_G(G; \LLambda^{q})$: the assignment
  \begin{equation*}
    x \in G \longmapsto u(x) \in \C_G(G; \LLambda^{q})
  \end{equation*}
  is obviously smooth -- hence defines an element $u \in \cinfty(G; \LLambda^{q})$ --, and satisfies $\mathsf{T}(u(x)) = f(x)$ for every $x \in G$. By the definition of $\mathsf{S}$, the coefficients of $u$ with respect to the basis $\tau_J$ with $|J|=q$ are linear combinations with complex coefficients of the coefficients $f$, and since $f \in \K(G; \LLambda^{q+1})$ the coefficients of $u$ are also in $\K(G)$, and, therefore, $u \in \K(G; \LLambda^{q})$. We conclude from~\eqref{imagemdedlinha} that $\dd' u = f$. 
\end{proof}

\begin{Rem}[A representation formula for $H^q_{\dd'}(\K(G))$]
  \label{rem:rep-formula-tensor-K}
  Notice that, for each $q$,
  \begin{equation}
    \label{eq:K-as-tensor}
    \K(G; \LLambda^q) \cong \K(G) \otimes \C_G(G; \LLambda^{q})
  \end{equation}
  (this and other isomorphisms below are meant as topological linear). Indeed, since $\C_G(G; \LLambda^{q})$ is a finite dimensional vector space, we may take $\theta_1^{(q)}, \ldots, \theta_{N_q}^{(q)}$ a linear basis for it and represent every element in $\K(G) \otimes \C_G(G; \LLambda^{q})$ uniquely as
  \begin{equation*}
    \sum_{j = 1}^{N_q} f_j \otimes \theta_j^{(q)},
    \quad f_j \in \K(G),
  \end{equation*}
  to which we assign the $q$-form $\sum_{j = 1}^{N_q} f_j \theta_j^{(q)} \in \K(G; \LLambda^q)$. This assignment is obviously linear, continuous, injective and, moreover, surjective: any $u \in \K(G; \LLambda^q)$ is written as~\eqref{eq:urepglobal} with $u_J \in \K(G)$ each ordered multi-index $J = (j_1, \ldots, j_q)$; writing each $\tau_J \in \C_G(G; \LLambda^{q})$ as a linear combination of $\theta_1^{(q)}, \ldots, \theta_{N_q}^{(q)}$ settles the claim. Continuity of its inverse is evident.

  Also, as far as~\eqref{eq:K-as-tensor} is concerned, $\dd': \K(G; \LLambda^q) \to \K(G; \LLambda^{q + 1})$ is represented as
  \begin{equation*}
    1 \otimes \dd': \K(G) \otimes \C_G(G; \LLambda^{q}) \lra \K(G) \otimes \C_G(G; \LLambda^{q + 1}),
  \end{equation*}
  thanks to the obvious identity $\dd' (f \theta) = f \dd' \theta$ whenever $f \in \K(G)$; in that sense, \eqref{eq:K-as-tensor} is a chain map, hence an isomorphism between differential complexes. Thus:
  \begin{align*}
    H^q_{\dd'}(\K(G))
    &\cong
      \frac{ \ker \{ 1 \otimes \dd': \K(G) \otimes \C_G(G; \LLambda^{q}) \lra \K(G) \otimes \C_G(G; \LLambda^{q + 1}) \} }
      { \ran \{ 1 \otimes \dd': \K(G) \otimes \C_G(G; \LLambda^{q - 1}) \lra \K(G) \otimes \C_G(G; \LLambda^{q}) \} } \\
    &=
      \frac{ \K(G) \otimes \ker \{ \dd': \C_G(G; \LLambda^{q}) \lra \C_G(G; \LLambda^{q + 1}) \} }
      { \K(G) \otimes \ran \{ \dd': \C_G(G; \LLambda^{q - 1}) \lra  \C_G(G; \LLambda^{q}) \} }.
  \end{align*}
  Finally, it follows from~\cite[Proposition~2.18]{am_ica} that exactness of the short sequence of vector spaces
  \begin{multline*}
    \ran \{ \dd': \C_G(G; \LLambda^{q - 1}) \to \C_G(G; \LLambda^{q}) \} \xrightarrow{\text{incl.}} \\
    \ker \{ \dd': \C_G(G; \LLambda^{q}) \to \C_G(G; \LLambda^{q + 1}) \} \xrightarrow{\text{proj.}}
    H^q_{\dd'}(\C_G) \lra 0
  \end{multline*}
  entails exactness of
  \begin{multline*}
    \K(G) \otimes \ran \{ \dd': \C_G(G; \LLambda^{q - 1}) \to \C_G(G; \LLambda^{q}) \} \xrightarrow{1 \otimes \text{incl.}} \\
    \K(G) \otimes \ker \{ \dd': \C_G(G; \LLambda^{q}) \to \C_G(G; \LLambda^{q + 1}) \}  \xrightarrow{1 \otimes \text{proj.}}
    \K(G) \otimes H^q_{\dd'}(\C_G) \lra 0,
  \end{multline*}
  where, clearly, all the spaces involved are Fr{\'e}chet and the arrows are continuous. The last arrow $1 \otimes \text{proj.}$~in the sequence above is onto, and descends to the quotient as a continuous linear isomorphism
  \begin{equation*}
    \frac{ \K(G) \otimes \ker \{ \dd': \C_G(G; \LLambda^{q}) \lra \C_G(G; \LLambda^{q + 1}) \} }
    { \K(G) \otimes \ran \{ \dd': \C_G(G; \LLambda^{q - 1}) \lra \C_G(G; \LLambda^{q}) \} }
    \cong
    \K(G) \otimes H^q_{\dd'}(\C_G)
  \end{equation*}
  whose inverse is continuous by the Open Mapping Theorem. Together, both isomorphisms prove that
  \begin{equation*}
    H^q_{\dd'}(\K(G)) \cong \K(G) \otimes H^q_{\dd'}(\C_G).
  \end{equation*}
\end{Rem}

Next, we delve into the structure of the second factor in~\eqref{eq:direct_sum_cohomology}.
\begin{Prop}
  \label{prop:RbarR}
  Let $q \in \{1, \ldots, n\}$. If $f \in \RR(G; \LLambda^q)$ is such that $\dd' f = 0$ then there exists $v \in \bar{\RR}(G; \LLambda^{q - 1})$ such that $\dd' v = f$. That is, the natural map
  \begin{equation}
    \label{eq:RbarR}
    H^q_{\dd'}(\RR(G)) \lra H^q_{\dd'}(\bar{\RR}(G))
  \end{equation}
  is zero.
\end{Prop}
\begin{proof}
  Write a $\dd'$-closed $f \in \RR(G; \LLambda^{q})$ in terms of coefficients $f_J \in \RR(G)$. For each $J$, we have $f_J = \Delta_{\gr{v}} g_J$ for some $g_J \in \cinfty(G)$. Replacing $g_J$ by $\pi_{\bar{\RR}} (g_J)$, we may assume  without loss of generality that $g_J \in \bar{\RR}(G)$, and define
  \begin{equation*}
    g \dfn \psum_{|J| = q} g_J \ \tau_J \in \bar{\RR}(G; \LLambda^{q}),
  \end{equation*}
  which obviously solves $\Delta_{\gr{v}} g = f$. Moreover
  \begin{equation*}
    0
    =  \dd' f
    = \dd' (\Delta_{\gr{v}} g)
    = \Delta_{\gr{v}} (\dd' g),
  \end{equation*}
  which implies that $\dd' g \in \K(G; \LLambda^{q + 1})$. Since also $\dd' g \in \bar{\RR}(G; \LLambda^{q + 1})$, we conclude that $\dd' g = 0$, and therefore Lemma~\ref{lem:ce} ensures the existence of a $v \in \bar{\RR}(G; \LLambda^{q - 1})$ solving $\dd' v = \Delta_{\gr{v}} g = f$.
\end{proof}

\section{Globally solvable structures}
\label{sec:GlobalSolvableStructures}

\begin{Thm}
  \label{thm:rep_in_K}
  If $\dd': \cinfty(G) \to \cinfty(G; \LLambda^{1})$ has closed range then
  \begin{equation*}
    H_{\dd'}^q (\bar{\RR}(G)) = 0,
    \quad \forall q \in \{0, \ldots, n\},
  \end{equation*}
  and in that case the natural map
  \begin{equation}
    \label{eq:Kintocinfty_cohom}
    H^q_{\dd'}(\K(G)) \lra H^q_{\dd'}(\cinfty(G))
  \end{equation}
  is a topological isomorphism; in particular, $H^q_{\dd'}(\cinfty(G))$ is a Fr{\'e}chet space. Moreover, every $[f] \in H^{q}_{\dd'}(\cinfty(G))$ has a representative $\pi_{\K} (f) \in \K(G; \LLambda^{q})$.
\end{Thm}
\begin{proof}
  On the one hand, we know from Proposition~\ref{prop:RbarR} that the natural map~\eqref{eq:RbarR} is always zero. On the other hand, the hypothesis of closedness of the range of $\dd'$ is equivalent (as already mentioned~\cite[Proposition~3.11]{afr22}) to the closedness of~\eqref{eq:Delta_range} in $\cinfty(G)$, i.e.,~$\RR(G) = \bar{\RR}(G)$, in which case~\eqref{eq:RbarR} is the identity: the first conclusion follows at once.

  Now, by the direct sum decomposition~\eqref{eq:direct_sum_cohomology}, our first conclusion entails that the map~\eqref{eq:Kintocinfty_cohom} is onto; we already knew that it was a continuous injection. Hence, \eqref{eq:Kintocinfty_cohom} is a continuous linear isomorphism. Its inverse is precisely the map induced by $\pi_{\K}$ in cohomology, hence continuous as well.
  
  The final claim is simply a restatement of the surjectivity of~\eqref{eq:Kintocinfty_cohom}.
\end{proof}
\begin{Cor}
  \label{cor:hausdorff_cohomology_all}
  If $\dd': \cinfty(G) \to \cinfty(G; \LLambda^{1})$ has closed range then so does $\dd': \cinfty(G; \LLambda^{q - 1}) \to \cinfty(G; \LLambda^{q})$ for every $q \in \{2, \ldots, n\}$.
\end{Cor}
\begin{Cor}
  \label{cor:GH}
  If $\VV$ is globally hypoelliptic then every class in $H^{q}_{\dd'}(\cinfty(G))$ has a representative in $\C_G(G; \LLambda^{q})$.
\end{Cor}
\begin{proof}
  The hypothesis is equivalent to global hypoellipticity of the first $\dd'$ in the differential complex associated to $\VV$; that is, any $f \in \D'(G)$ satisfying $\dd' f \in \cinfty(G; \LLambda^{1})$ is already smooth to start with. This implies~\cite[Theorem~2.2]{afjr24} that the map $\dd': \cinfty(G) \to \cinfty(G; \LLambda^{1})$ has closed range; a standard functional analytic argument further shows that its kernel $\K(G)$ is of finite dimension. Since the latter is a subalgebra of $\cinfty(G)$ (by Leibniz rule), it must be equal to $\C_{G}$. The conclusion will follow from Theorem~\ref{thm:rep_in_K}.

  Indeed, let $f \in \K(G)$. Hence, $f^n \in \K(G)$ for every $n \in \N$, and since $\K(G)$ is finite dimensional there exists $n_0 \in \N$ such that $1, f, f^2, \ldots, f^{n_0}$ are linearly dependent. This entails the existence of a polynomial $P \in \C[z]$ such that $P(f) = 0$, meaning that $f(G)$ consists of roots of $P$, which are finitely many. By connectedness, $f(G)$ is a point, i.e.,~$f$ is constant.
\end{proof}

It is a natural question to ask whether some kind of converse to Corollary~\ref{cor:hausdorff_cohomology_all} holds. We provide an affirmative answer assuming that $\mathfrak{v}$ is commutative. Notice that this assumption does not preclude non-commutativity of $G$.

\begin{Thm}
  \label{thm:abelian-converse}
  Assume that $\mathfrak{v}$ is Abelian. If $\dd': \cinfty(G; \LLambda^{q - 1}) \to \cinfty(G; \LLambda^{q})$ has closed range for some $q \in \{2, \ldots, n\}$, then $\dd':\cinfty(G) \to \cinfty(G; \LLambda^1)$ has closed range. 
\end{Thm}
\begin{proof}
  The key algebraic aspect of Abelian structures is the simplification of the expression~\eqref{eq:isom_dprime3} of $\dd'$ thanks to the vanishing of the structural constants $\alpha_{JK}$. Indeed, by~\eqref{eq:structure-ctes2} we have, for an ordered multi-index $K = (k_0, \ldots, k_q)$:
  \begin{equation*}
    \alpha_{JK} = (\dd' \tau_J)(\LL_{k_0}, \ldots, \LL_{k_q}) = 0
  \end{equation*}
  thanks to~\eqref{eq:invariant_formula_dprime}, the duality relations between $\tau_1, \ldots, \tau_n$ and $\LL_1, \ldots, \LL_n$, and the fact that the latter is a family of pairwise commuting vector fields (our hypothesis). Therefore, if we write $u \in \cinfty(G; \LLambda^q)$ in its canonical form~\eqref{eq:urepglobal}, then
  \begin{equation}
    \label{eq:isom_dprime_comm}
    \dd' u
    =
    \sum_{j = 1}^n \psum_{|J| = q} (\vv{L}_j u_{J}) \ \tau_j \wedge \tau_J.
  \end{equation}

  The proof of the theorem starts by assuming, otherwise, that the range of $\Delta_{\gr{v}}: \cinfty(G) \to \cinfty(G)$ is not closed. By results in~\cite{araujo19}, given $C, \rho > 0$ there exist $\lambda \in \sigma(\Delta_G)$ and $\phi \in E_\lambda \cap \left( \ker \Delta_{\gr{v}} \right)^\perp$ such that $\| \phi \|_{L^2(G)} = 1$ and
  \begin{equation}
    \label{falhaemresolubilidadeglobal}
    \Big( \sum_{j=1}^n \| \LL_j \phi \|_{L^2(G)}^2 \Big)^{\frac{1}{2}} < C (1 + \lambda)^{-\rho}.  
  \end{equation}
  (see the proof of Proposition~\ref{prop:direct_sum_closed_range} for notation). It follows that for each $\nu \in \N$ we can find $0 < C_\nu \leq 1$, $\lambda_\nu \in \sigma(\Delta_G)$ and $\phi_\nu \in E_{\lambda_\nu} \cap \left(\ker \Delta_{\gr{v}} \right)^\perp$ with $\| \phi_\nu \|_{L^2(G)} = 1$ and
  \begin{equation}
    \label{eq:solvability-ineq}
    \Big( \sum_{j=1}^n \|\LL_j \phi_\nu \|_{L^2(G)}^2 \Big)^{\frac{1}{2}} < C_\nu (1 + \lambda_\nu) ^{-\nu},
    \quad \forall \nu \in\N.
  \end{equation}
  Concerning this construction, below we deduce a few simplifications that we will be able to assume without loss of generality.
  
  First, we may further assume $\{ \lambda_\nu \}_{\nu \in \N}$ strictly increasing. Indeed, putting $C = 1$ and $\rho = 1$ in~\eqref{falhaemresolubilidadeglobal}, we find $\lambda_1 \in \sigma(\Delta_G)$ and $\phi_1 \in E_{\lambda_1} \cap (\ker \Delta_{\mathfrak{v}})^\perp$ such that
  \begin{equation*}
    \Big(\sum_{j=1}^n \| \LL_j\phi_1 \|_{L^2(G)}^2 \Big)^{\frac{1}{2}} < (1 + \lambda_1)^{-1},
  \end{equation*}
  and pick $C_1 \dfn 1$. Since $\sigma(\Delta_G)$ is discrete and each $E_\lambda$ is finite dimensional, the set
  \begin{equation*}
    \mathcal{S}_{\lambda_1} \dfn
    \bigcup_{\substack{\lambda\in \sigma(\Delta_G) \\ \lambda\leq \lambda_1}}
    \left\{ \phi \in E_{\lambda} \cap (\ker \Delta_{\mathfrak{v}} )^\perp \st  \| \phi \|_{L^2(G)} = 1 \right\}
  \end{equation*}
  is compact, so by~\eqref{eq:delta-vs-fields-norm} we have
  \begin{equation*}
    m_{\lambda_1} \dfn \min_{\phi \in \mathcal{S}_{\lambda_1}} \Big(\sum_{j=1}^n \| \LL_j \phi \|_{L^2(G)}^2 \Big)^{\frac{1}{2}} > 0.
  \end{equation*}
  Picking $0 < C_2< \min\{ m_{\lambda_1}, 1 \}$ and putting $C = C_2$ and $\rho = 2$ in~\eqref{falhaemresolubilidadeglobal}, we obtain $\lambda_2 \in \sigma(\Delta_G)$, and $\phi_2 \in E_{\lambda_2} \cap (\ker \Delta_{\mathfrak{v}})^\perp$ with $\|\phi_2 \|_{L^{2}(G)} = 1$ such that
  \begin{equation*}
    \Big( \sum_{j=1}^n \| \LL_j \phi_2 \|_{L^2(G)}^2 \Big)^{\frac{1}{2}} < C_2 (1 + \lambda_2 )^{-2}.
  \end{equation*}
  Note that $\lambda_2 > \lambda_1$: otherwise, $\phi_2 \in \mathcal{S}_{\lambda_1}$, hence
  \begin{equation*}
    C_2 (1 + \lambda_2 )^{-2} < m_{\lambda_1} \leq \Big(\sum_{j=1}^n \| \LL_j \phi_2 \|_{L^2(G)}^2 \Big)^{\frac{1}{2}}
  \end{equation*}
  leading to a contradiction. Proceeding in this way we obtain the desired sequences.

  Fix $\nu \in \N$. Since the vector fields $\LL_1,\ldots, \LL_n$ restrict to $E_{\lambda_\nu}$ as commuting skew-adjoint linear endomorphisms, we can pick an orthonormal basis
  \begin{equation*}
    \mathcal B_\nu \dfn \left\{ \varphi_{1}^{(\nu)}, \ldots, \varphi_{d_\nu}^{(\nu)} \right\}
  \end{equation*}
  of $E_{\lambda_\nu}$ formed by common eigenfunctions of these operators, where $d_{\nu} \dfn \dim E_{\lambda_\nu}$. After relabeling, we can assume that
  \begin{equation*}
    \mathcal{B}_\nu'
    \dfn \{\varphi \in \mathcal B_\nu \st \Delta_{\gr{v}} \varphi \neq 0 \}
    = \left\{ \varphi_{1}^{(\nu)}, \ldots, \varphi_{q_\nu}^{(\nu)} \right\}
  \end{equation*}
  for some $q_\nu \in \{ 1, \ldots, d_\nu\}$ -- $\mathcal{B}_\nu' \neq \eset$ since $\phi_\nu \in E_{\lambda_\nu}$. Noting that $\varphi_{1}^{(\nu)}, \ldots, \varphi_{d_\nu}^{(\nu)}$ are also eigenfunctions of $\Delta_{\mathfrak{v}}$ -- thanks to its very definition~\eqref{eq:DeltaV} --, elementary arguments show that $\mathcal B_\nu'$ is an orthonormal basis of $E_{\lambda_\nu} \cap (\ker \Delta_{\mathfrak{v}})^\perp$.
  
  
  Thus $\phi_\nu \in \Span_\C \mathcal{B}_\nu'$, but we may also assume that actually $\phi_\nu \in \mathcal{B}_\nu'$. Indeed, otherwise we would have
  \begin{equation*}
    \Big( \sum_{j=1}^n \|\LL_j \varphi_\ell^{(\nu)} \|_{L^2(G)}^2\Big)^{\frac{1}{2}} \geq C_\nu (1 + \lambda_\nu )^{-\nu},
    \quad \forall \ell \in \{1, \ldots, q_\nu\}.
  \end{equation*}
  Writing $\LL_j \varphi_\ell^{(\nu)} = \mu_{j, \ell}^{(\nu)} \varphi_\ell^{(\nu)}$, we have, for any $\phi \in E_{\lambda_\nu} \cap (\ker \Delta_{\mathfrak{v}})^\perp$:
  \begin{align*}
    \sum_{j=1}^n \|\LL_j \phi \|_{L^2(G)}^2
    &= \sum_{j=1}^n \Big\| \sum_{\ell=1}^{q_\nu} \langle \phi, \varphi_\ell^{(\nu)} \rangle_{L^2(G)} \ \LL_j \varphi_\ell^{(\nu)} \Big\|_{L^2(G)}^2 \\
    &= \sum_{j=1}^n \Big\| \sum_{\ell=1}^{q_\nu} \langle \phi, \varphi_\ell^{(\nu)} \rangle_{L^2(G)} \ \mu_{j, \ell}^{(\nu)} \varphi_\ell^{(\nu)} \Big\|_{L^2(G)}^2 \\
    &= \sum_{j=1}^n \sum_{\ell=1}^{q_\nu} |\langle \phi, \varphi_\ell^{(\nu)} \rangle_{L^2(G)} |^2 \ | \mu_{j, \ell}^{(\nu)}|^2 \\
    &= \sum_{j=1}^n \sum_{\ell=1}^{q_\nu} |\langle \phi, \varphi_\ell^{(\nu)} \rangle_{L^2(G)} |^2 \ \| \LL_j \varphi_\ell^{(\nu)} \|_{L^2(G)}^2 \\
    &\geq \sum_{\ell=1}^{q_\nu} |\langle \phi, \varphi_\ell^{(\nu)} \rangle_{L^2(G)} |^2 \ C_\nu^2 (1 + \lambda_\nu )^{-2\nu} \\
    &= \| \phi \|_{L^2(G)}^2 C_\nu^2 (1 + \lambda_\nu )^{-2\nu},
  \end{align*}
  which, for $\phi = \phi_\nu$, contradicts~\eqref{eq:solvability-ineq}. Hence, we may assume that $\phi_\nu \in \mathcal{B}_\nu'$, and, after further relabeling, that $\phi_\nu = \varphi_1^{(\nu)}$.
  
  Finally, we can assume without loss of generality that
  \begin{equation}
    \label{eq:solvability-ineq-max}
    \max_{1 \leq j \leq n} \|\LL_j \phi_\nu \|_{L^2(G)} = \|\LL_1 \phi_\nu \|_{L^2(G)} ,
    \quad \forall \nu \in \N.
  \end{equation}
  Indeed, the maximum in the left-hand side of~\eqref{eq:solvability-ineq-max} must be attained infinitely many times at some $j_0 \in \{1, \ldots, n\}$. Relabeling the vector fields so that as to have $j_0 = 1$ and extracting the underlying subsequence in $\nu$ does the job. In particular, $\LL_1 \phi_\nu \neq 0$ for all $\nu \in \N$.
  
  Now we set
  \begin{equation*}
    f \dfn \sum_{\nu=1}^\infty \sum_{j=1}^n (\LL_j \phi_\nu) \ \tau_j \wedge \tau_2 \wedge \cdots \wedge \tau_{q+1}
    \in \cinfty(G; \LLambda^{q + 1}).
  \end{equation*}
  Indeed, $\mathcal{F}_{\lambda}(f) = 0$ if $\lambda \neq \lambda_\nu$ for all $\nu \in \N$, whereas
  \begin{equation*}
    \| \mathcal{F}_{\lambda_\nu}(f) \|_{L^2(G)}^2
    = \| \LL_1 \phi_\nu \|_{L^2(G)}^2 + \sum_{j = q + 2}^n \| \LL_j \phi_\nu \|_{L^2(G)}^2
    < (1 + \lambda_\nu)^{-2\nu},
    \quad \forall \nu \in \N,
  \end{equation*}
  by~\eqref{eq:solvability-ineq}, from which smoothness of $f$ follows by~\eqref{eq:paley-wiener}, which also entails convergence of the series above in $\cinfty(G; \LLambda^{q + 1})$. Moreover, $f$ belongs to the closure of the range of $\dd': \cinfty(G; \LLambda^q) \to \cinfty(G; \LLambda^{q + 1})$: notice that, thanks to~\eqref{eq:isom_dprime_comm},
  \begin{equation*}
    f = \lim_{N \to\infty} \dd' \left( \sum_{\nu=1}^N \phi_\nu \ \tau_2 \wedge \cdots \wedge \tau_{q+1} \right).
  \end{equation*}
  But by hypothesis the said range is closed, so we obtain $u \in \cinfty (G; \LLambda^q)$ such that $\dd' u = f$. We have, on the one hand,
  \begin{equation*}
    f(\LL_1, \LL_2, \ldots, \LL_{q + 1}) = \sum_{\nu=1}^\infty \LL_1 \phi_\nu
  \end{equation*}
  by definition of $f$; on the other hand, by~\eqref{eq:invariant_formula_dprime} (recalling, once more, that $\gr{v}$ is commutative) we have
  \begin{align*}
    (\dd' u) (\LL_1, \LL_2, \ldots, \LL_{q + 1})
    &= \sum_{j = 1}^{q + 1} (-1)^{j + 1} \LL_j \left( u (\LL_1, \ldots, \widehat{\LL}_j, \ldots, \LL_{q + 1}) \right) \\
    &= \sum_{j = 1}^{q + 1} (-1)^{j + 1} \LL_j u_{\widehat{J}(j)}
  \end{align*}
  with coefficients $u_J$ as in~\eqref{eq:urepglobal}, where $\widehat{J}(j) \dfn (1, \ldots, \widehat{j}, \ldots, q + 1)$. We conclude that
  \begin{equation*}
    \sum_{j = 1}^{q + 1} (-1)^{j + 1} \LL_j u_{\widehat{J}(j)} = \sum_{\nu=1}^\infty \LL_1 \phi_\nu
  \end{equation*}
  which we project onto $E_{\lambda_\nu}$, obtaining
  \begin{align*}
    \LL_1 \varphi_1^{(\nu)}
    &= \LL_1 \phi_\nu \\
    &= \sum_{j = 1}^{q + 1} (-1)^{j + 1} \mathcal{F}_{\lambda_\nu} (\LL_j u_{\widehat{J}(j)}) \\
    &= \sum_{j = 1}^{q + 1} (-1)^{j + 1} \LL_j \mathcal{F}_{\lambda_\nu} (u_{\widehat{J}(j)}) \\
    &= \sum_{j = 1}^{q + 1} \sum_{\ell = 1}^{q_\nu} (-1)^{j + 1} \langle u_{\widehat{J}(j)}, \varphi_\ell^{(\nu)} \rangle_{L^2(G)} \ \LL_j \varphi_\ell^{(\nu)}.
  \end{align*}
  Recalling that the $\varphi_\ell^{(\nu)}$ are eigenfunctions of each $\LL_j$, by the linear independence of $\varphi_1^{(\nu)}, \ldots, \varphi_{q_\nu}^{(\nu)}$ it follows that
  \begin{equation*}
    \LL_1 \varphi_1^{(\nu)} = \LL_1 \phi_\nu = \sum_{j = 1}^{q + 1} (-1)^{j + 1} \langle u_{\widehat{J}(j)}, \varphi_1^{(\nu)} \rangle_{L^2(G)} \ \LL_j \varphi_1^{(\nu)}
  \end{equation*}
  hence computing $L^2$ norms yields, using~\eqref{eq:solvability-ineq-max},
  \begin{align*}
    1
    &= \sum_{j = 1}^{q + 1} |\langle u_{\widehat{J}(j)}, \varphi_1^{(\nu)} \rangle_{L^2(G)}|^2 \ \frac{\| \LL_j \varphi_1^{(\nu)} \|_{L^2(G)}^2}{\| \LL_1 \varphi_1^{(\nu)} \|_{L^2(G)}^2} \\
    &\leq \sum_{j = 1}^{q + 1} |\langle u_{\widehat{J}(j)}, \varphi_1^{(\nu)} \rangle_{L^2(G)}|^2 \\
    &\leq \sum_{j = 1}^{q + 1} \| \mathcal{F}_{\lambda_\nu} (u_{\widehat{J}(j)}) \|_{L^2(G)}^2.
  \end{align*}
  which is impossible since $u_{\widehat{J}(j)} \in \cinfty(G)$ implies that $\| \mathcal{F}_{\lambda_\nu} (u_{\widehat{J}(j)}) \|_{L^2(G)}$ must decay rapidly in $\lambda_\nu$ (in the sense of~\eqref{eq:paley-wiener}).
\end{proof}

\section{Structures whose integral subgroup is closed}
\label{sec:integral}

Since $\gr{v}$ is essentially real, there exists a Lie subalgebra $\gr{h} \sset \gr{g}$ whose complexification equals $\gr{v}$: we denote by $H \sset G$ its (connected) integral subgroup, which, in this section, we will further assume to be closed. In that case, the space $G/H$ of left cosets is a smooth, compact manifold, and the natural projection $\Pi: G \to G/H$ is a smooth submersion. We have
\begin{equation}
  \label{eq:scriptK_when_H_closed}
  \K(G) = \{ \Pi^* g \st  g \in \cinfty(G/H) \} \cong \cinfty(G/H).
\end{equation}

\begin{Thm}
  \label{thm:subgroup_closed}
  If $H \sset G$ is closed then   $\dd': \cinfty(G) \to \cinfty(G; \LLambda^1)$ has closed range.
\end{Thm}
\begin{proof}
  Pick $\LL_1, \ldots, \LL_n$ an orthonormal basis of $\gr{h}$: as in Section~\ref{sec:decompositions} we will prove, equivalently, that $\Delta_{\gr{v}}: \cinfty(G) \to \cinfty(G)$ has closed range.
  
  We endow $H$ with the $\ad$-invariant metric inherited from $G$ (which is the same as endowing $\gr{h} \sset \gr{g}$ with the induced inner product), and regard $\LL_1, \ldots, \LL_n$ as left-invariant vector fields on $H$ as well, in which sense
  \begin{equation*}
    \Delta_H = - \sum_{j = 1}^n \LL_j^2
  \end{equation*}
  is the Laplace-Beltrami operator on $H$ associated to the metric. It follows, since the vector fields $\LL_1, \ldots, \LL_n$ are tangent to $H$, that $(\Delta_{\gr{v}} f)|_H = \Delta_H (f|_H)$ for every $f \in \cinfty(G)$. The following estimates for $\Delta_H$ (which is elliptic) are well-known:
  \begin{equation*}
    \| h \|_{L^2(H)} \leq C  \| \Delta_H h \|_{L^2(H)},
    \quad \forall h \in \cinfty(H), \ h \perp_{L^2(H)} \C_H.
  \end{equation*}
  We will also make use of the following Fubini formula~\cite[Theorem~8.36]{knapp_lgbi}:
  \begin{equation*}
    \int_G f(x) \ \dd \mu_G(x)
    =
    \int_{G/H} \left( \int_H f(xy) \ \dd \mu_H(y) \right) \dd \mu_{G/H} (xH),
    \quad \forall f \in \cinfty(G),
  \end{equation*}
  where $\mu_G, \mu_H$ are the respective Haar measures, and $\mu_{G/H}$ is a certain $G$-invariant measure on $G/H$. Notice that $x \mapsto \int_H f(xy) \dd \mu_H(y)$ descends to $G/H$ thanks to the left-invariance of $\mu_H$.

  Let $f \in \cinfty(G)$ be $L^2$-orthogonal to $\K(G)$. By~\eqref{eq:scriptK_when_H_closed} we have
  \begin{equation*}
    0
    = \langle f, \Pi^* g \rangle_{L^2(G)}
    =  \int_{G/H} \overline{g(xH)} \left( \int_H f(xy) \ \dd \mu_H(y) \right) \dd \mu_{G/H} (xH),
    \quad \forall g \in \cinfty(G/H),
  \end{equation*}
  hence, for every $x \in G$,
  \begin{equation*}
    \langle (\ell_x^* f)|_H, 1_H \rangle_{L^2(H)}
    =
    \int_H f(xy) \ \dd \mu_H(y)
    = 0.
  \end{equation*}
  This shows that $(\ell_{x}^{\ast} f)|_H \perp_{L^2(H)} \C_H$, and thus 
  \begin{equation*}
    \| (\ell_x^* f)|_H \|_{L^2(H)}
    \leq
    C  \| \Delta_H [(\ell_x^* f)|_H] \|_{L^2(H)},
    \quad \forall x \in G.
  \end{equation*}
  Now notice that $\Delta_H [(\ell_x^* f)|_H] = [\Delta_{\gr{v}} (\ell_x^* f)]|_H = [ \ell_x^* (\Delta_{\gr{v}} f)]|_H$, and therefore
  \begin{align*}
    \| f \|_{L^2(G)}^2
    &= \int_G |f(x)|^2 \dd \mu_G(x) \\
    &= \int_{G/H} \left( \int_H |f(xy)|^2 \dd \mu_H(y) \right) \dd \mu_{G/H} (xH) \\
    &= \int_{G/H}  \| (\ell_x^* f)|_H \|_{L^2(H)}^2 \ \dd \mu_{G/H} (xH) \\
    &\leq C^2 \int_{G/H} \| [ \ell_x^* (\Delta_{\gr{v}} f)]|_H \|_{L^2(H)}^2  \ \dd \mu_{G/H} (xH) \\
    &= C^2 \| \Delta_{\gr{v}} f \|_{L^2(G)}^2.
  \end{align*}
  
  We have proved that
  \begin{equation*}
    \| f \|_{L^2(G)} \leq C  \| \Delta_{\gr{v}} f \|_{L^2(G)},
    \quad \forall f \in \cinfty(G), \ f \perp_{L^2(G)} \K(G).
  \end{equation*}
  Since $\Delta_{\gr{v}}$ is a left-invariant operator, our conclusion follows from results in~\cite{araujo19}.
\end{proof}
\begin{Cor}
  \label{cor:subgroup_closed_iso}
  If $H \sset G$ is closed then
  \begin{equation*}
    H^{q}_{\dd'}(\cinfty(G)) \cong \cinfty(G/H) \otimes H^q_{\mathrm{dR}} (H),
    \quad \forall q \in \{0, \ldots, n\}.
  \end{equation*}
\end{Cor}
\begin{proof}
  By Theorems~\ref{thm:subgroup_closed} and~\ref{thm:rep_in_K}, Remark~\ref{rem:rep-formula-tensor-K} and equation~\eqref{eq:scriptK_when_H_closed}, we have
  \begin{equation*}
    H^{q}_{\dd'}(\cinfty(G)) \cong H^{q}_{\dd'}(\K(G)) \cong \K(G) \otimes H^{q}_{\dd'}(\C_G) \cong \cinfty(G/H) \otimes H^{q}_{\dd'}(\C_G),
  \end{equation*}
  so the only thing left to prove is that $H^{q}_{\dd'}(\C_{G}) \cong H^q_{\mathrm{dR}} (H)$.

  Pick $\LL_1, \ldots, \LL_n, \MM_1, \ldots, \MM_m$ a basis of $\gr{g}$ such that $\gr{h} = \Span \{\LL_1, \ldots, \LL_n\}$, and $\tau_1, \ldots, \tau_n, \zeta_1, \ldots, \zeta_m \in \gr{g}^*$ the dual basis; we will regard them as global frames for $TG$ and $T^*G$, respectively. Let $\imath: H \hookrightarrow G$ be the inclusion map: denoting by $\LL_1^\bb, \ldots, \LL_n^\bb$ the vector fields induced on $H$, they form a global frame of $TH$, and $\LL_j^\bb$ is $\imath$-related to $\LL_j$ for each $j \in \{1, \ldots, n\}$. It follows that $\imath^* \zeta_k = 0$ for each $k \in \{1, \ldots, m\}$, since for each $p \in H$ we have
  \begin{equation*}
    \langle \imath^* \zeta_k|_p, \LL_j^\bb|_p \rangle
    =
    \langle \zeta_k|_{\imath(p)}, \imath_*(\LL_j^\bb|_p) \rangle
    =
    \langle \zeta_k|_{\imath(p)}, \LL_j|_{\imath(p)} \rangle
    =
    0,
    \quad \forall j \in \{1, \ldots, n\}.
  \end{equation*}
  A similar argument shows that $\imath^* \tau_1, \ldots, \imath^* \tau_n$ is a global frame for $T^*H$, dual to $\LL_1^\bb, \ldots, \LL_n^\bb$. Therefore, writing an arbitrary $u \in \cinfty(G; \LLambda^q)$ as~\eqref{eq:urepglobal}, we have
  \begin{equation*}
    \imath^* u = \psum_{|J| = q} (\imath^* u_J) \ \imath^* \tau_J.
  \end{equation*}
  Since $\imath^*: \cinfty(G) \to \cinfty(H)$ is surjective (because $H \sset G$ is a closed embedded submanifold), it follows that $\imath^*: \cinfty(G; \LLambda^q) \to \cinfty(H; \wedge^q \C T^* H)$ is also surjective,  with kernel
  \begin{equation*}
    \{ u \in \cinfty(G; \LLambda^q) \st u_J|_H = 0, \quad \forall J \}.
  \end{equation*}
  Moreover, since $\imath^* \zeta_k = 0$ for all $k \in \{1, \ldots, m\}$, it follows that $\imath^*(\dd u - \dd' u) = 0$, thus
  \begin{equation*}
    \imath^* (\dd' u) = \imath^* (\dd u) = \dd \imath^* u,
  \end{equation*}
  i.e.,~$\imath^*: \cinfty(G; \LLambda^q) \lra \cinfty(H; \wedge^q \C T^* H)$ is a chain map. In particular, it induces maps in cohomology
  \begin{equation*}
    \imath^*: H^q_{\dd'}(\cinfty(G)) \lra H^q_{\mathrm{dR}}(H).
  \end{equation*}
  
  By restriction, it is immediate that $\imath^*: \C_G(G; \LLambda^q) \lra \C_H(H; \wedge^q \C T^* H)$ is an isomorphism (where the target is the space $\Span_\C \{ \imath^* \tau_J \st |J| = q \}$ of left-invariant $q$-forms on $H$). Hence,
  \begin{equation*}
    \imath^*: H^q_{\dd'}(\C_{G}) \lra H^q_{\dd}(\C_H)
  \end{equation*}
  is also an isomorphism, where $H^q_{\dd}(\C_H)$ stands for left-invariant de Rham cohomology of $H$. Finally, $H^q_{\dd}(\C_H) \cong H^q_{\mathrm{dR}}(H)$ (by~\cite{ce48}, or our Theorem~\ref{thm:rep_in_K} plus Remark~\ref{rem:rep-formula-tensor-K} applied to the de Rham cohomology of $H$).
\end{proof}

\section{Tube structures and related equations}
\label{sec:tube}

In this final section, we address some questions raised in a previous paper~\cite{afjr24} in the context of tube structures.

Let $M$ be a smooth $n$-dimensional manifold, whose bundle of $q$-forms we denote by $\Lambda^q$; the pullback of the latter to $\Omega \dfn M \times \TT^m$ will be denoted, in this section, by $\LLambda^q$: the bundle of \emph{$(0,q)$-forms}. The choice of $m$ smooth, closed $1$-forms on $M$ -- call them $\omega_1, \ldots, \omega_m$ -- induced on $\Omega$ an involutive structure $\VV \sset \C T \Omega$, often called a \emph{tube structure} on $\Omega$. Its associated differential complex~\cite{bch_iis, treves_has} can be concretely written as
\begin{equation*}
  \dd' \dfn \dd_t + \sum_{k = 1}^m \omega_k \wedge \del_{x_k}:
  \cinfty(\Omega; \LLambda^q) \lra \cinfty(\Omega; \LLambda^{q + 1}),
  \quad q \in \{0, \ldots, n - 1 \},
\end{equation*}
with cohomology spaces (based on smooth sections) denoted by $H^q_{\dd'}(\cinfty(\Omega))$.

\subsection{Normalization}
\label{subsec:normalization}

Consider on $\Omega$ two tube structures $\VV, \VV^\bb \sset \C T \Omega$, associated to $1$-forms $\omega_1, \ldots, \omega_m$ and $\omega_1^\bb, \ldots, \omega_m^\bb$, respectively.
\begin{Thm}
  Suppose there exist $\psi_1, \ldots, \psi_m \in \cinfty(M; \R)$ such that
  \begin{equation*}
    \omega_k^\bb = \omega_k + \dd \psi_k,
    \quad k \in \{1, \ldots, m\}.
  \end{equation*}
  Then there exists a diffeomorphism $\Psi: \Omega \to \Omega$ such that
  \begin{equation}
    \label{eq:compatible}
    \Psi_*(\VV_p) = \VV^\bb_{\Psi(p)},
    \quad \forall p \in \Omega.
  \end{equation}
\end{Thm}
\begin{proof}
  We define $\Psi: M \times \TT^m \to M \times \TT^m$ by
  \begin{equation*}
    \Psi(t, x_1, \ldots, x_m) \dfn (t, x_1 + \psi_1(t), \ldots, x_m + \psi_m(t)),
  \end{equation*}
  where the sum $x_k + \psi_k(t)$ takes place in $\R / 2 \pi \Z$. This is clearly a diffeomorphism (its inverse has the same form). To prove~\eqref{eq:compatible}, we show instead the dual relationship between the annihilator bundles:
  \begin{equation*}
    \Psi^* (\T'^\bb_{\Psi(p)}) = \T'_p,
    \quad \forall p \in \Omega.
  \end{equation*}
  
  Let $p = (t,x) \in \Omega$, and pick a neighborhood of $\Psi(p) = (t, x^\bb)$ of the form $U^\bb \times V^\bb$, where:
  \begin{itemize}
  \item $U^\bb \sset M$ is a neighborhood of $t$ where one can find $\phi_1^\bb, \ldots, \phi_m^\bb \in \cinfty(U^\bb)$ local primitives for $\omega_1^\bb, \ldots, \omega_m^\bb$, respectively; and
  \item $V^\bb \sset \TT^m$ is a neighborhood of $x^\bb$ with standard coordinates $(x_1^\bb, \ldots, x_m^\bb)$.
  \end{itemize}
  Take now $U \times V \sset \Psi^{-1}(U^\bb \times V^\bb)$ a neighborhood of $p$, where $U \sset M$ is a neighborhood of $t$ and $V \sset \TT^m$ is a neighborhood of $x$ with standard coordinates $(x_1, \ldots, x_m)$. Using the definition of $\Psi$, one checks at once that $U \sset U^\bb$, hence $\phi_k \dfn \phi_k^\bb - \psi_k$ is a primitive for $\omega_k$ on $U$ for each $k \in \{1, \ldots, m\}$.

  This means that
  \begin{equation*}
    \zeta_k \dfn \dd x_k - \dd \phi_k,
    \quad k \in \{1, \ldots, m\},
  \end{equation*}
  span $\T'$ on $U \times V$, whereas
  \begin{equation*}
    \zeta_k^\bb \dfn \dd x_k^\bb - \dd \phi_k^\bb,
    \quad k \in \{1, \ldots, m\},
  \end{equation*}
  span $\T'^\bb$ on $U^\bb \times V^\bb$. We have, on $U \times V$:
  \begin{equation*}
    \Psi^* x_k^\bb = x_k^\bb \circ \Psi = x_k + \psi_k + 2\pi \nu_k,
    \quad
    \Psi^* \phi_k^\bb = \phi_k^\bb \circ \Psi = \phi_k^\bb
  \end{equation*}
  (recall that $\phi_k^\bb$ depends only on the $t$-variables), for some $\nu \in \Z^m$ fixed. Hence,
  \begin{equation*}
    \Psi^* \zeta_k^\bb
    = \dd \Psi^* x_k^\bb - \dd \Psi^* \phi_k^\bb
    = \dd (x_k + \psi_k + 2 \pi \nu_k) - \dd \phi_k^\bb
    = \zeta_k
  \end{equation*}
  for each $k \in \{1, \ldots, m\}$, thus proving our assertion.
\end{proof}

\begin{Cor}
  \label{cor:real_normalization}
  Assume further that all $\omega_k, \omega_k^\bb$ are real-valued. If $[\omega_k] = [\omega_k^\bb]$ in $H^1_{\mathrm{dR}}(M)$ for every $k \in \{1, \ldots, m\}$, then there exists a diffeomorphism $\Psi: \Omega \to \Omega$ such that~\eqref{eq:compatible} holds.
\end{Cor}

Diffeomorphisms like $\Psi$ above are said to be \emph{compatible} with the involutive structures $\VV, \VV^\bb$. Thanks to~\eqref{eq:compatible}, they induce pullback maps $\Psi^*$ on the level of $(0,q)$-forms that commute with the associated differentials $\dd', \dd'_\bb$ (i.e.,~$\Psi^*$ is a chain map), hence conjugating the complexes by means of topological isomorphisms, and inducing isomorphism in cohomology:
\begin{equation*}
  \Psi^*: H^{q}_{\dd'_\bb}(\cinfty(\Omega)) \lra H^{q}_{\dd'}(\cinfty(\Omega)).
\end{equation*}
In particular, all the properties we are concerned with are preserved by it.

Since every closed $1$-form on a compact, connected Lie group $M$ admits a left-invariant representative (by~\cite{ce48}, or our Theorem~\ref{thm:rep_in_K} plus Remark~\ref{rem:rep-formula-tensor-K} applied to the de Rham cohomology of $M$), we have the following special case of Corollary~\ref{cor:real_normalization}:
\begin{Cor}
  \label{cor:normalization_lie_groups}
  When $M$ is a compact, connected Lie group, every real tube structure on $G \dfn M \times \TT^m$ is equivalent to one defined by left-invariant $1$-forms.
\end{Cor}

\subsection{A formula for cohomology}
\label{sec:formula-tube}

In this section we assume that $M$ is a compact, connected $n$-dimensional manifold, and start by recalling some context from~\cite{afjr24}. There, we defined, for each $\xi \in \Z^m$ and $q \in \{0, \ldots, n - 1\}$, a differential operator $\dd'_\xi : \cinfty(M; \Lambda^q) \to \cinfty(M; \Lambda^{q + 1})$ by
\begin{equation*}
  \dd'_\xi f \dfn \dd f + i \sum_{k = 1}^m \xi_k \omega_k \wedge f.
\end{equation*}
It follows that $\dd'_\xi \circ \dd'_\xi = 0$, and the associated cohomology spaces
\begin{equation*}
  H^q_\xi(\cinfty(M))
  \dfn
  \frac{ \ker \{ \dd'_\xi: \cinfty(M; \Lambda^{q}) \lra \cinfty(M; \Lambda^{q + 1}) \}}
       { \ran \{ \dd'_\xi: \cinfty(M; \Lambda^{q - 1}) \lra \cinfty(M; \Lambda^{q}) \}}
\end{equation*}
are finite dimensional thanks to ellipticity of that complex, plus compactness of $M$. We further let $\sol_\xi$ be the sheaf of germs of smooth functions on $M$ annihilated by $\dd'_\xi$, and considered the group~\cite[Theorem~3.6]{afjr24}
\begin{equation*}
  \Gamma_{\pmb{\omega}}
  \dfn
  \{ \xi \in \Z^m \st \text{$\sol_\xi$ has a non-vanishing global section} \}.
\end{equation*}
Then, we turned our attention to the property:
\begin{equation}
  \label{eq:vanishing_mysterious_cohomologies}
  H^q_\xi(\cinfty(M)) = \{0\},
  \quad \forall \xi \in \Z^m \setminus \Gamma_{\pmb{\omega}}.
\end{equation}
When~\eqref{eq:vanishing_mysterious_cohomologies} holds (which is always the case, e.g.,~when $M$ is the $2$-sphere, or the $2$-torus), we proved that, for instance~\cite[Theorem~7.5]{afjr24}:
\begin{equation*}
  \text{$H_{\dd'}^{q}(\cinfty(\Omega))$ is finite dimensional}
  \Longleftrightarrow
  H_{\dd'}^{q}(\cinfty(\Omega)) \cong H^q_{\mathrm{dR}}(M).
\end{equation*}
When $\omega_1, \ldots, \omega_m$ are \emph{real}, a more precise result can be obtained~\cite[Theorem~5.7]{afjr24}: if $\dd': \cinfty(\Omega; \LLambda^{q - 1}) \to \cinfty(\Omega; \LLambda^{q})$ has closed range and~\eqref{eq:vanishing_mysterious_cohomologies} holds, then there is a natural isomorphism
\begin{equation*}
  H^q_{\dd'} (\cinfty(\Omega)) \cong \cinfty(\TT^r) \otimes H^q_{\mathrm{dR}}(M),
\end{equation*}
where $r$ is the rank of the group $\Gamma_{\pmb{\omega}}$. By the end of the next section, we will have proved the following result:
\begin{Thm}
  \label{thm:mysterious_cohomology_left_invariant_lis}
  If $M$ is a compact, connected \emph{Lie group} and $\omega_1, \ldots, \omega_m$ are closed, \emph{left-invariant} and real-valued then~\eqref{eq:vanishing_mysterious_cohomologies} holds for every $q \in \{0, \ldots, n\}$.
\end{Thm}
Actually, thanks to Corollary~\ref{cor:normalization_lie_groups}, the hypothesis of left-invariance of $\omega_1, \ldots, \omega_m$ is inessential. Indeed, for each $k \in \{1, \ldots, m\}$ let $\omega_k^\bb \dfn \omega_k + \dd \psi_k$ for some $\psi_k \in \cinfty(M; \R)$, and consider $\dd''_\xi$ the operators associated with the new system $\omega_1^\bb, \ldots, \omega_m^\bb$. That is:
\begin{equation*}
  \dd''_\xi f \dfn \dd f + i \sum_{k = 1}^m \xi_k \omega_k^\bb \wedge f.
\end{equation*}
A simple computation shows that
\begin{equation*}
  \dd'_\xi (e^{i \phi} f) = e^{i \phi} \dd''_\xi f,
  \quad \text{where $\phi \dfn \sum_{k = 1}^m \xi_k \psi_k$},
\end{equation*}
hence both complexes (and all their associated constructions: sheaves of solutions, cohomologies, etc.)~are isomorphic: for real $\omega_1, \ldots, \omega_m$, condition~\eqref{eq:vanishing_mysterious_cohomologies} also depends only on their de Rham cohomology classes. The conclusions will, then, be summarized as follows:
\begin{Cor}
  \label{cor:mysterious_cohomology_left_invariant_lis}
  If $M$ is a compact, connected  Lie group, $\omega_1, \ldots, \omega_m$ are closed and real-valued, and $r$ is the rank of the group $\Gamma_{\pmb{\omega}}$, then
  \begin{equation*}
    H^q_{\dd'} (\cinfty(\Omega)) \cong \cinfty(\TT^r) \otimes H^q_{\mathrm{dR}}(M)
  \end{equation*}
  holds for every  $q \in \{0, \ldots, n\}$.
\end{Cor}

\subsection{Perturbations of the de Rham complex}
\label{sec:perturb}

Motivated by the discussion in the previous section, we consider a class of zero-order perturbations of the de Rham complex of $M$. Given $\omega$ a smooth, closed $1$-form on $M$, for each $q \in \{0, \ldots, n\}$ we define a first-order operator acting on $q$-forms:
\begin{equation*}
  \DD_\omega \dfn \dd + i \omega \wedge \cdot :
  \cinfty(M; \Lambda^q) \lra \cinfty(M; \Lambda^{q + 1}).
\end{equation*}
Clearly $\DD_\omega \circ \DD_\omega = 0$, and the differential complex so constructed is elliptic. Its cohomology spaces are finite dimensional, and denoted by
\begin{equation}
  \label{eq:cohomology-spaces-Domega}
  H_\omega^q(\cinfty(M)) \dfn
  \frac{\ker \{ \DD_\omega: \cinfty(M; \Lambda^q) \lra \cinfty(M; \Lambda^{q + 1}) \}}
  {\ran \{ \DD_\omega: \cinfty(M; \Lambda^{q - 1}) \lra \cinfty(M; \Lambda^q) \}}.
\end{equation}
We will prove:
\begin{Thm}
  \label{theremkerneliszerothencohomologyiszero}
  If $M$ is a compact, connected Lie group, and $\omega$ is closed, left-invariant and real-valued, then we have two possibilities:
  \begin{enumerate}
  \item either
    \begin{equation*}
      \ker \{ \DD_\omega: \cinfty(M) \lra \cinfty(M; \Lambda^{1}) \} = \{0\},
    \end{equation*}
    in which case $H_\omega^q(\cinfty(M)) = \{0\}$ for every $q \in \{1, \ldots, n \}$;
  \item or the kernel of $\DD_\omega$ is non-trivial, in which case $H_\omega^q(\cinfty(M)) \cong H^q_{\mathrm{dR}}(M)$ for every $q \in \{1, \ldots, n\}$.
  \end{enumerate}
\end{Thm}

\begin{Rem}
  The algebraic parts of the proofs below, including the ones in the appendix, are inspired by an argument in~\cite[pp.~117--118]{ce48}.
\end{Rem}

We start by giving explicit global descriptions of the objects involved. Let $\vv{T}_1, \ldots, \vv{T}_n$ be a basis for the Lie algebra $\gr{m}$ of $M$, with dual basis  $\tau_1, \ldots, \tau_n \in \gr{m}^*$. We write
\begin{equation*}
  \omega = \sum_{j = 1}^n a_j \tau_j,
  \quad a_j \in \R,
\end{equation*}
due to its left-invariance. An arbitrary $u \in \cinfty(M; \Lambda^q)$  is written as
\begin{equation}
  \label{eq:forms_M_lie_group}
  u = \psum_{|J| = q} u_{J} \ \tau_J,
  \quad u_{J} \in \cinfty(M),
\end{equation}
hence
\begin{equation*}
  \DD_\omega u
  = \psum_{|J| = q} \sum_{j = 1}^n  \left( \vv{T}_j u_{J} + i a_j u_{J} \right) \ \tau_j \wedge \tau_J
  + \psum_{|J| = q} u_{J} \ \dd \tau_J,
\end{equation*}
that is,  $\DD_\omega$ is expressed in global frames as a matrix of left-invariant differential operators on $M$. For that reason, it allows for Fourier analysis.

Take $\Delta_M$ the Laplace-Beltrami operator associated to some $\ad$-invariant metric on $M$. For each $\mu \in \sigma(\Delta_M)$, we denote by $\cinfty_{E_\mu}(M)$ the eigenspace associated to $\mu$, and by $\cinfty_{E_\mu}(M; \Lambda^{q})$ the space of $q$-forms~\eqref{eq:forms_M_lie_group} such that $u_J \in \cinfty_{E_\mu}(M)$ for every $J$. It follows that the entries of $\DD_\omega$ commute with $\Delta_M$, hence $\DD_\omega \cinfty_{E_\mu}(M; \Lambda^{q}) \sset \cinfty_{E_\mu}(M; \Lambda^{q + 1})$ for every $\mu$.

By finite dimensionality of $H^q_\omega(\cinfty(M))$, it follows from~\cite[Corollary~2.9]{araujo19} that
\begin{equation}
  \label{eq:agag_iso}
  H^q_\omega(\cinfty(M))
  \cong
  \bigoplus_{\mu \in \sigma(\Delta_M)}
  \frac{ \ker \{ \DD_\omega: \cinfty_{E_\mu}(M; \Lambda^{q}) \lra \cinfty_{E_\mu}(M; \Lambda^{q+1}) \}}
  { \ran \{ \DD_\omega: \cinfty_{E_\mu}(M; \Lambda^{q-1}) \lra \cinfty_{E_\mu}(M; \Lambda^{q}) \}}.
\end{equation}
In particular, at most finitely many quotients in the right-hand side of~\eqref{eq:agag_iso} are non-trivial.
\begin{proof}[Proof of Theorem~\ref{theremkerneliszerothencohomologyiszero}]
  First we assume that there is a non-zero $\phi \in \cinfty(M)$ such that $\dd \phi = -i \phi \omega$. If $\phi$ did vanish at some $t_0\in M$, we could take $U \sset M$ a connected open neighborhood of $t_0$ where $\omega$ is exact: $\omega|_U = \dd \psi$ for some $\psi \in \cinfty(U)$. Therefore:
  \begin{equation*}
    \dd ( \phi e^{i \psi}) = \dd \phi \wedge e^{ i \psi} + i \phi e^{i \psi} \wedge \dd \psi = -i \phi \omega \wedge  e^{ i \psi} + i\phi e^{i \psi} \wedge \omega=0.
  \end{equation*}
  Consequently, $\phi e^{i \psi}$ is constant, and since $\phi(t_0) = 0$ it follows that $\phi$ vanishes on $U$. Since $M$ is connected, this leads to a contradiction, so $\phi$ cannot vanish anywhere in $M$. Thus we can write
  \begin{equation*}
    \dd (f / \phi) = (1 / \phi) \DD_\omega (f),
    \quad \forall f \in \cinfty(M; \Lambda^q),
  \end{equation*}
  hence multiplication by $\phi^{-1}$ yields $H_\omega^q(\cinfty(M)) \cong H^q_{\mathrm{dR}}(M)$ for every $q \in \{0, \ldots, n\}$.

  Next we consider the case where $\DD_\omega$ is injective. Notice that  $\cinfty_{E_\mu}(M; \Lambda^{q})$  is finite dimensional, just as the  cohomology spaces
  \begin{equation*}
    H^q_\omega(\cinfty_{E_\mu}(M)) \dfn
    \frac{ \ker \{ \DD_\omega: \cinfty_{E_\mu}(M; \Lambda^{q}) \lra \cinfty_{E_\mu}(M; \Lambda^{q+1}) \}}
         { \ran \{ \DD_\omega:  \cinfty_{E_\mu}(M; \Lambda^{q-1}) \lra \cinfty_{E_\mu}(M; \Lambda^{q})\}},
         \quad \mu \in \sigma(\Delta_M),
  \end{equation*}
  which, as we have seen, vanish except for finitely many $\mu$: thanks to~\eqref{eq:agag_iso}, in order to prove the vanishing of~\eqref{eq:cohomology-spaces-Domega} we must check that $H^q_\omega(\cinfty_{E_\mu}(M)) = \{0\}$ for \emph{all} $\mu$.
  
  On $G \dfn M \times \TT^1$, we consider $\VV \sset \C T G$ the corank $1$ tube structure determined by the single $1$-form $\dd x - \omega$. Since $\omega$ is left-invariant, so is $\VV$, with associated subalgebra $\gr{v} \sset \C \gr{g}$ spanned by the real left-invariant vector fields
  \begin{equation*}
    \LL_j \dfn \vv{T}_j + a_j \del_x,
    \quad j \in \{1, \ldots, n\}.
  \end{equation*}
  Fix $\mu \in \sigma(\Delta_M)$ and consider the space
  \begin{equation*}
    \mathscr{F}_{\mu}(G)
    \dfn
    \{ e^{ix} \phi \st \phi \in \cinfty_{E_\mu}(M) \}
    \sset \cinfty(G),
  \end{equation*}
  which is finite dimensional and $\gr{v}$-invariant:
  \begin{equation*}
    \vv{L}_j (e^{ix} \phi )
    = e^{ix} \left( \vv{T}_j \phi + i a_j \phi \right),
    \quad \forall j \in \{1, \ldots, n\}.
  \end{equation*}
  Notice also that
  \begin{equation*}
    \mathscr{F}_{\mu} (G; \LLambda^{q}) =  \{ e^{ix} u \st u \in \cinfty_{E_\mu}(M; \Lambda^q) \} \sset \cinfty(G; \LLambda^q)
  \end{equation*}
  is finite dimensional for every $q\in \{0, \ldots, n\}$, and for $u \in \cinfty_{E_\mu}(M; \Lambda^{q})$ we have
  \begin{equation*}
    \dd' ( e^{i x} u)
    = (\dd_t + \omega \wedge \del_x)(e^{i x} u)
    = e^{i x} \dd_t u + i e^{i x} \omega \wedge u
    = e^{ix} (\DD_\omega u).
  \end{equation*}
  The map induced by multiplication by $e^{i x}$   in cohomology
  \begin{equation*}
    H^q_\omega(\cinfty_{E_\mu}(M)) \lra H^{q}_{\dd'}(\mathscr{F}_{\mu}(G))
  \end{equation*}
  is thus well-defined, and moreover injective. Indeed, if $[f] \in H^q_\omega(\cinfty_{E_\mu}(M))$ is such that $[e^{ix} f] = 0$ in $H^{q}_{\dd'}(\mathscr{F}_{\mu}(G))$ then there exists $u \in \mathscr{F}_{\mu} (G; \LLambda^{q - 1})$ such that  $\dd' u = e^{ix} f$. As such, there exists $v \in \cinfty_{E_\mu}(M; \Lambda^{q - 1})$ such that $u = e^{i x} v$, hence
  \begin{equation*}
    e^{ix}f = \dd' ( e^{ix} v ) = e^{ix} ( \DD_\omega v )
    \Longrightarrow
    f = \DD_\omega v
  \end{equation*}
  from which $[f] = 0$ follows. To finish, we prove that $H^{q}_{\dd'}(\mathscr{F}_{\mu}(G)) = \{0\}$.
    
  First, notice that $\Delta_{\gr{v}}: \mathscr{F}_{\mu} (G; \LLambda^{q}) \lra \mathscr{F}_{\mu} (G; \LLambda^{q})$ is a bijection; since $\mathscr{F}_{\mu} (G; \LLambda^{q})$ is finite dimensional, it suffices to check injectivity. Take $f \in \mathscr{F}_{\mu} (G; \LLambda^{q})$ such that $\Delta_{\gr{v}} f = 0$. We have
  \begin{equation*}
    f = \psum_{|J| = q} f_{J} \ \tau_J,
    \quad f_{J} \in \mathscr{F}_{\mu}(G)
    \Longrightarrow
    \Delta_{\gr{v}} f = \psum_{|J| = q} (\Delta_{\gr{v}} f_{J}) \ \tau_J,
  \end{equation*}
  hence $\Delta_{\gr{v}} f_{J} = 0$ for each $J$. However, for an arbitrary $e^{ix} \phi \in \mathscr{F}_{\mu}(G)$ we have
  \begin{align*}
    \Delta_{\gr{v}} (e^{ix} \phi) = 0
    &\Longleftrightarrow
      \vv{L}_j (e^{i x} \phi) = 0,
      \quad \forall j \in \{1, \ldots, n\} \\
    &\Longleftrightarrow
      \dd' (e^{ix} \phi) = 0 \\
    &\Longleftrightarrow
      \DD_\omega \phi = 0 \\
    &\Longleftrightarrow
      \phi = 0
  \end{align*}
  by hypothesis. So  $f_J = 0$  for all $J$,  i.e.,~$f = 0$.
 
  Finally, let $f \in \mathscr{F}_{\mu} (G; \LLambda^{q})$ be such that $\dd' f = 0$, and take $g \in \mathscr{F}_{\mu} (G; \LLambda^{q})$ such that $\Delta_{\gr{v}} g = f$. Then
  \begin{equation*}
    0 = \dd' f = \dd' \Delta_{\gr{v}} g = \Delta_{\gr{v}} (\dd' g),
  \end{equation*}
  which implies that $\dd' g = 0$ by the injectivity of $\Delta_{\gr{v}}: \mathscr{F}_{\mu} (G; \LLambda^{q + 1}) \to \mathscr{F}_{\mu} (G; \LLambda^{q + 1})$ shown above. Lemma~\ref{lem:ce} ensures the existence of a $v \in \mathscr{F}_{\mu} (G; \LLambda^{q - 1})$ such that $\dd' v = \Delta_{\gr{v}} g = f$. Since $f$ is arbitrary, we conclude that $H^{q}_{\dd'}(\mathscr{F}_{\mu}(G)) = \{0\}$.
\end{proof}

Back to Theorem~\ref{thm:mysterious_cohomology_left_invariant_lis}, fix $\xi\in \Z^{m}$ and $\omega_1, \ldots, \omega_m$ left-invariant $1$-forms. For
\begin{equation*}
  \omega \dfn \sum_{k=1}^{m} \xi_k \omega_k
\end{equation*}
we have $\dd'_\xi= \DD_{\omega}$ and $H^q_{\xi} (\cinfty(M))= H^{q}_\omega(\cinfty(M))$. Note that $\xi \notin \Gamma_{\pmb{\omega}}$ if and only if $\ker \{\DD_\omega: \cinfty(M) \to \cinfty(M; \Lambda^{1}) \} = \{0\}$. Therefore, Theorem~\ref{theremkerneliszerothencohomologyiszero} implies that condition~\eqref{eq:vanishing_mysterious_cohomologies} holds for every $q\in \{0, \ldots, n\}$.

\appendix

\section{Proofs of the algebraic lemmas}
\label{ap:algebraic_proof}

\begin{proof}[Proof of Lemma~\ref{lem:lap_commutes}]
  Since $\gr{v} = \bar{\gr{v}}$, it is enough to prove the statement assuming $\vv{L}$ real. As differential operators on $G$, we have
  \begin{equation*}
    - [\vv{L}, \Delta_{\gr{v}}]
    = \sum_{j = 1}^n [\vv{L}, \vv{L}_j^2]
    = \sum_{j = 1}^n [\vv{L}, \vv{L}_j] \vv{L}_j + \vv{L}_j [\vv{L}, \vv{L}_j].
  \end{equation*}
  Moreover, since $[\vv{L}, \vv{L}_j] \in \gr{v}$ we can write
  \begin{equation}
    \label{eq:orth_comm}
    [\vv{L}, \vv{L}_j] = \sum_{k = 1}^n c_{jk} \vv{L}_k
  \end{equation}
  where
  \begin{equation}
    \label{eq:orth_comm_coeff}
    c_{jk}
    \dfn \langle [\vv{L}, \vv{L}_j], \vv{L}_k \rangle
    = - \langle \vv{L}_j, [\vv{L}, \vv{L}_k] \rangle
    = - \langle [\vv{L}, \vv{L}_k], \vv{L}_j \rangle
    = - c_{kj}
  \end{equation}
  (by $\ad$-invariance; we make use of our assumption that $\vv{L}$ is real). Hence:
  \begin{align*}
    - [\vv{L}, \Delta_{\gr{v}}]
    &= \sum_{j = 1}^n \sum_{k = 1}^n c_{jk} \vv{L}_k \vv{L}_j
    + \sum_{j = 1}^n \sum_{k = 1}^n c_{jk} \vv{L}_j \vv{L}_k \\
    &= \sum_{j = 1}^n \sum_{k = 1}^n c_{jk} \vv{L}_k \vv{L}_j
    - \sum_{j = 1}^n \sum_{k = 1}^n c_{kj} \vv{L}_j \vv{L}_k \\
    &= 0.
  \end{align*}
\end{proof}

\begin{Lem}
  \label{lem:fund_identity}
  For every $\vv{L} \in \gr{v}$ and every $1$-form $f$ on $G$ we have
  \begin{equation}
    \label{eq:rel_ce}
    \sum_{j = 1}^n [\vv{L}, \vv{L}_j] f(\vv{L}_j) = - \sum_{j = 1}^n \vv{L}_j f([\vv{L}, \vv{L}_j]).
  \end{equation}
\end{Lem}
\begin{proof}
  By linearity, once more it is enough to assume $\vv{L}$ real. Writing $[\vv{L}, \vv{L}_j]$ as in~\eqref{eq:orth_comm}, one develops both sides of~\eqref{eq:rel_ce}: they match perfectly thanks to~\eqref{eq:orth_comm_coeff}.
\end{proof}

Before we state the key technical lemma, we recall a few operations on forms, as well as their relationships. Given a vector field $\vv{L} \in \gr{v}$ and $u$ a $q$-form we denote by $\imath_{\vv{L}} u$ their contraction, which is the$(q - 1)$-form defined by
\begin{equation*}
  (\imath_{\vv{L}} u)(\vv{X_1}, \ldots, \vv{X}_{q - 1}) \dfn u(\vv{L}, \vv{X_1}, \ldots, \vv{X}_{q - 1}),
  \quad \text{$\vv{X_1}, \ldots, \vv{X}_{q - 1}$ vector fields on $G$},
\end{equation*}
hence $\imath_{\vv{L}} \cinfty(G; \T'^{1, q - 1}) \sset \cinfty(G; \T'^{1, q - 2})$. Moreover, $\imath_{\vv{L}} \cinfty(G; \LLambda^{q}) \sset \cinfty(G; \LLambda^{q - 1})$ since $\imath_{\vv{L}}$ is $\cinfty(G)$-linear, and, by~\cite[eqn.~(13.3)]{lee_smooth},
\begin{equation*}
  \imath_{\vv{L}} (\tau_{j_1} \wedge \cdots \wedge \tau_{j_q})
  = \sum_{k = 1}^q (-1)^{k - 1} \tau_{j_k} (\vv{L})
  \tau_{j_1} \wedge \cdots \wedge \widehat{\tau}_{j_k} \wedge \cdots \wedge \tau_{j_q}
  \in \C_G(G; \LLambda^{q - 1}).
\end{equation*}
Actually, $\imath_{\vv{L}} \mathscr{F}(G; \LLambda^{q}) \sset \mathscr{F}(G; \LLambda^{q - 1})$ for every $\gr{v}$-invariant subspace $\mathscr{F}(G) \sset \cinfty(G)$.

For $u \in \cinfty(G; \LLambda^{q})$, the Lie derivative $\mathcal{L}_{\vv{L}} u$ satisfies~\cite[Proposition~18.9(f)]{lee_smooth}
\begin{align*}
  (\mathcal{L}_{\vv{L}} u) (\pmb{\vv{X}})
  &= \mathcal{L}_{\vv{L}} \left( u(\pmb{\vv{X}}) \right)
  + \sum_{k = 1}^q (-1)^{k} u(\mathcal{L}_{\vv{L}} \vv{X}_k, \widehat{\pmb{\vv{X}}}_k) \\
  &= \vv{L} \left( u(\pmb{\vv{X}}) \right)
  + \sum_{k = 1}^q (-1)^{k} u([\vv{L}, \vv{X}_k], \widehat{\pmb{\vv{X}}}_k),
\end{align*}
where $\pmb{\vv{X}} \dfn (\vv{X}_1, \ldots, \vv{X}_q)$ and $\widehat{\pmb{\vv{X}}}_k \dfn (\vv{X}_1, \ldots, \widehat{\vv{X}}_k, \ldots, \vv{X}_q)$ for $k \in \{1, \ldots, m\}$. This shows that $\mathcal{L}_{\vv{L}} \cinfty(G; \T'^{1, q - 1}) \sset \cinfty(G; \T'^{1, q - 1})$, in particular inducing a linear map $\mathcal{L}_{\vv{L}}':  \cinfty(G; \LLambda^{q}) \to \cinfty(G; \LLambda^{q})$. It satisfies, for $\vv{X}_1, \ldots, \vv{X}_q$ sections of $\VV$:
\begin{equation*}
  (\mathcal{L}_{\vv{L}}' u) (\pmb{\vv{X}})
  = \vv{L} \left( u(\pmb{\vv{X}}) \right)
  + \sum_{k = 1}^q (-1)^{k} u([\vv{L}, \vv{X}_k], \widehat{\pmb{\vv{X}}}_k),
\end{equation*}
whatever $u \in \cinfty(G; \LLambda^{q})$. Again, we have $\mathcal{L}_{\vv{L}}' \mathscr{F}(G; \LLambda^{q}) \sset \mathscr{F}(G; \LLambda^{q})$ for any $\gr{v}$-invariant subspace $\mathscr{F}(G) \sset \cinfty(G)$. Moreover, Cartan's formula~\cite[Proposition~18.13]{lee_smooth}
\begin{equation*}
  \mathcal{L}_{\vv{L}} u = \imath_{\vv{L}} (\dd u) + \dd (\imath_{\vv{L}} u)
\end{equation*}
yields, by the properties just established,
\begin{equation*}
  \mathcal{L}'_{\vv{L}} u + \cinfty(G; \T'^{1, q - 1})
  =
  \imath_{\vv{L}} (\dd' u) + \cinfty(G; \T'^{1, q - 1})
  +
  \dd' (\imath_{\vv{L}} u) + \cinfty(G; \T'^{1, q - 1}),
\end{equation*}
which translates into
\begin{equation*}
  \mathcal{L}'_{\vv{L}} u = \imath_{\vv{L}} (\dd' u) + \dd' (\imath_{\vv{L}} u),
  \quad \forall u \in \cinfty(G; \LLambda^{q}).
\end{equation*}

\begin{Prop}
  \label{prop:ce}
  If $f \in \cinfty(G; \LLambda^{q})$ is such that $\dd' f = 0$ then
  \begin{equation*}
    v \dfn \sum_{j = 1}^n \vv{L}_j (\imath_{\vv{L}_j} f)
  \end{equation*}
  solves $\dd' v = -\Delta_{\gr{v}} f$.
\end{Prop}
\begin{Rem}
  For $\mathscr{F}(G)$ a $\gr{v}$-invariant subspace, $f \in \mathscr{F}(G; \LLambda^{q})$ implies $v \in \mathscr{F}(G; \LLambda^{q - 1})$.
\end{Rem}
\begin{proof}
  We write
  \begin{equation}
    \label{eq:first_identity_technical_lemma}
    \dd' v
    = \sum_{j = 1}^n \dd' \vv{L}_j (\imath_{\vv{L}_j} f)
    = \sum_{j = 1}^n \left( \dd' \vv{L}_j (\imath_{\vv{L}_j} f)
    - \vv{L}_j \dd' (\imath_{\vv{L}_j} f)  \right)
    + \sum_{j = 1}^n \vv{L}_j \dd' (\imath_{\vv{L}_j} f),
  \end{equation}
  and focus on the first sum, which we develop by evaluating at $\vv{X}_1, \ldots, \vv{X}_q \in \gr{v}$. Thanks to~\eqref{eq:comm_A_dprime}, we have
  \begin{align*}
    \left( \dd' \vv{L}_j (\imath_{\vv{L}_j} f) - \vv{L}_j \dd' (\imath_{\vv{L}_j} f) \right) (\pmb{\vv{X}})
    &= \sum_{k = 1}^q (-1)^{k + 1} [\vv{X}_k, \vv{L}_j]
    \left( (\imath_{\vv{L}_j} f) (\widehat{\pmb{\vv{X}}}_k) \right) \\
    &= \sum_{k = 1}^q (-1)^{k + 1} [\vv{X}_k, \vv{L}_j]
    f(\vv{L}_j, \widehat{\pmb{\vv{X}}}_k), 
  \end{align*}
  hence
  \begin{equation*}
    \sum_{j = 1}^n
    \left( \dd' \vv{L}_j (\imath_{\vv{L}_j} f) - \vv{L}_j \dd' (\imath_{\vv{L}_j} f) \right)
    (\pmb{\vv{X}})
    = \sum_{k = 1}^q (-1)^{k + 1}
    \sum_{j = 1}^n [\vv{X}_k, \vv{L}_j]
    f(\vv{L}_j, \widehat{\pmb{\vv{X}}}_k).
  \end{equation*}
  We then apply Lemma~\ref{lem:fund_identity} to the $1$-form $f(\cdot, \widehat{\pmb{\vv{X}}}_k)$ and $\vv{L} = \vv{X}_k$:
  \begin{equation*}
    \sum_{j = 1}^n [\vv{X}_k, \vv{L}_j] f(\vv{L}_j, \widehat{\pmb{\vv{X}}}_k)
    = \sum_{j = 1}^n \vv{L}_j f([\vv{L}_j, \vv{X}_k], \widehat{\pmb{\vv{X}}}_k),
  \end{equation*}
  which we replace in the previous identity to obtain
  \begin{equation*}
    \sum_{j = 1}^n
    \left( \dd' \vv{L}_j (\imath_{\vv{L}_j} f) - \vv{L}_j \dd' (\imath_{\vv{L}_j} f) \right)
    (\pmb{\vv{X}})
    = \sum_{k = 1}^q (-1)^{k + 1}
    \sum_{j = 1}^n \vv{L}_j
    f([\vv{L}_j, \vv{X}_k], \widehat{\pmb{\vv{X}}}_k).
  \end{equation*}

  Now we study the last sum in~\eqref{eq:first_identity_technical_lemma}:
  \begin{equation*}
    \sum_{j = 1}^n \vv{L}_j \dd' (\imath_{\vv{L}_j} f)
    = \sum_{j = 1}^n \vv{L}_j (\mathcal{L}'_{\vv{L}_j} f - \imath_{\vv{L}_j} (\dd' f))
    = \sum_{j = 1}^n \vv{L}_j \mathcal{L}'_{\vv{L}_j} f,
  \end{equation*}
  which we again evaluate at $\vv{X}_1, \ldots, \vv{X}_q \in \gr{v}$:
  \begin{equation*}
    (\mathcal{L}_{\vv{L}_j}' f) (\pmb{\vv{X}})
    = \vv{L}_j \left( f(\pmb{\vv{X}}) \right)
    + \sum_{k = 1}^q (-1)^{k} f([\vv{L}_j, \vv{X}_k], \widehat{\pmb{\vv{X}}}_k),
  \end{equation*}
  hence
  \begin{equation*}
    (\vv{L}_j \mathcal{L}_{\vv{L}_j}' f) (\pmb{\vv{X}})
    = \vv{L}_j \left( (\mathcal{L}_{\vv{L}_j}' f) (\pmb{\vv{X}}) \right) 
    = \vv{L}_j^2 \left( f(\pmb{\vv{X}}) \right)
    + \sum_{k = 1}^q (-1)^{k} \vv{L}_j f([\vv{L}_j, \vv{X}_k], \widehat{\pmb{\vv{X}}}_k),
  \end{equation*}
  and thus
  \begin{equation*}
    \sum_{j = 1}^n (\vv{L}_j \mathcal{L}'_{\vv{L}_j} f) (\pmb{\vv{X}})
    = -(\Delta_{\gr{v}} f)(\pmb{\vv{X}})
    + \sum_{k = 1}^q (-1)^{k}
    \sum_{j = 1}^n \vv{L}_j
    f([\vv{L}_j, \vv{X}_k], \widehat{\pmb{\vv{X}}}_k).
  \end{equation*}

  Summing everything up, all the inconvenient terms cancel out leaving us with
  \begin{equation*}
    (\dd' v)(\vv{X}_1, \ldots, \vv{X}_q)
    = -(\Delta_{\gr{v}} f)(\vv{X}_1, \ldots, \vv{X}_q),
    \quad \forall \vv{X}_1, \ldots, \vv{X}_q \in \gr{v},
  \end{equation*}
  which is enough to prove our assertion as forms in $\cinfty(G; \LLambda^{q})$ are completely determined by their action on vectors in $\gr{v}$.
\end{proof}

\def\cprime{$'$}


\begin{thebibliography}{10}

\bibitem{araujo19}
G.~Ara\'{u}jo.
\newblock Global regularity and solvability of left-invariant differential
  systems on compact {L}ie groups.
\newblock {\em Ann. Global Anal. Geom.}, 56(4):631--665, 2019.

\bibitem{afjr24}
G.~Ara\'{u}jo, I.~A. Ferra, M.~R. Jahnke, and L.~F. Ragognette.
\newblock Global solvability and cohomology of tube structures on compact
  manifolds.
\newblock {\em Math. Ann.}, 390(2):2199--2233, 2024.

\bibitem{afr22}
G.~Ara\'{u}jo, I.~A. Ferra, and L.~F. Ragognette.
\newblock Global solvability and propagation of regularity of sums of squares
  on compact manifolds.
\newblock {\em J. Anal. Math.}, 148(1):85--118, 2022.

\bibitem{am_ica}
M.~F. Atiyah and I.~G. Macdonald.
\newblock {\em Introduction to commutative algebra}.
\newblock Addison-Wesley Publishing Co., Reading, Mass.-London-Don Mills, Ont.,
  1969.

\bibitem{bcp96}
A.~P. Bergamasco, P.~D. Cordaro, and G.~Petronilho.
\newblock Global solvability for certain classes of underdetermined systems of
  vector fields.
\newblock {\em Math. Z.}, 223(2):261--274, 1996.

\bibitem{bpzz18}
A.~P. Bergamasco, A.~Parmeggiani, S.~L. Zani, and G.~A. Zugliani.
\newblock Geometrical proofs for the global solvability of systems.
\newblock {\em Math. Nachr.}, 291(16):2367--2380, 2018.

\bibitem{bp99sys}
A.~P. Bergamasco and G.~Petronilho.
\newblock Global solvability of a class of involutive systems.
\newblock {\em J. Math. Anal. Appl.}, 233(1):314--327, 1999.

\bibitem{bch_iis}
S.~Berhanu, P.~D. Cordaro, and J.~Hounie.
\newblock {\em An introduction to involutive structures}, volume~6 of {\em New
  Mathematical Monographs}.
\newblock Cambridge University Press, Cambridge, 2008.

\bibitem{ce48}
C.~Chevalley and S.~Eilenberg.
\newblock Cohomology theory of {L}ie groups and {L}ie algebras.
\newblock {\em Trans. Amer. Math. Soc.}, 63:85--124, 1948.

\bibitem{ckmt26}
S.~Coriasco, A.~Kirilov, W.~A.~A. de~Moraes, and P.~M. Tokoro.
\newblock Global hypoellipticity for involutive systems on non-compact
  manifolds.
\newblock {\em J. Geom. Anal.}, 36(1):Paper No. 22, 21, 2026.

\bibitem{dm16}
P.~L. Dattori~da Silva and A.~Meziani.
\newblock Cohomology relative to a system of closed forms on the torus.
\newblock {\em Math. Nachr.}, 289(17-18):2147--2158, 2016.

\bibitem{ds26}
P.~L. Dattori~da Silva and F.~M. Sim{\~a}o.
\newblock A differential complex on a compact {L}ie group.
\newblock {\em Ann. Mat. Pura Appl. (4)}, 2026.

\bibitem{hz17}
J.~Hounie and G.~Zugliani.
\newblock Global solvability of real analytic involutive systems on compact
  manifolds.
\newblock {\em Math. Ann.}, 369(3-4):1177--1209, 2017.

\bibitem{hz19}
J.~Hounie and G.~Zugliani.
\newblock Global solvability of real analytic involutive systems on compact
  manifolds. {P}art 2.
\newblock {\em Trans. Amer. Math. Soc.}, 371(7):5157--5178, 2019.

\bibitem{jahnke23}
M.~R. Jahnke.
\newblock Elliptic involutive structures on compact {L}ie groups.
\newblock {\em Ann. Sc. Norm. Super. Pisa Cl. Sci. (5)}, 24(1):487--518, 2023.

\bibitem{knapp_lgbi}
A.~W. Knapp.
\newblock {\em Lie groups beyond an introduction}, volume 140 of {\em Progress
  in Mathematics}.
\newblock Birkh\"auser Boston, Inc., Boston, MA, 1996.

\bibitem{lee_smooth}
J.~M. Lee.
\newblock {\em Introduction to smooth manifolds}, volume 218 of {\em Graduate
  Texts in Mathematics}.
\newblock Springer-Verlag, New York, 2003.

\bibitem{treves_has}
F.~Treves.
\newblock {\em Hypo-analytic structures}, volume~40 of {\em Princeton
  Mathematical Series}.
\newblock Princeton University Press, Princeton, NJ, 1992.
\newblock Local theory.

\end{thebibliography}
\end{document}